\documentclass[11pt,english]{article}

\usepackage[english]{babel}
\usepackage[utf8]{inputenc}
\usepackage[T1,OT1]{fontenc}
\usepackage{chngcntr}
\usepackage{amsmath,amssymb,amsthm,amsfonts,amsthm}
\usepackage{thmbox}
\usepackage{cases}
\usepackage{lipsum}
\usepackage{textcomp}
\usepackage{lmodern}
\usepackage{epsfig}
\usepackage[a4paper,textwidth=17cm,top=2.2cm,bottom=2.2cm]{geometry}
\usepackage[obeyspaces]{url}
\usepackage[a4paper,plainpages=false,hidelinks,bookmarks=false]{hyperref}
\usepackage{mathptmx}
\usepackage{verbatim}
\usepackage{stmaryrd}
\usepackage{calc}
\usepackage{fix-cm}
\usepackage{mdwlist}
\usepackage{enumerate}
\usepackage{enumitem}
\usepackage{floatflt}
\usepackage{mathrsfs}
\usepackage{caption}
\usepackage{subcaption}
\usepackage[affil-it]{authblk}
\theoremstyle{break}
\newtheorem[L,nocut]{defi}{Definition}
\newtheorem[L,nocut]{theo}{Theorem}
\newtheorem[L,nocut]{prop}{Proposition}
\newtheorem[L,nocut]{lemme}{Lemma}
\newtheorem[L,nocut]{cor}{Corollaire}
\newtheorem[L,nocut]{app}{Application}
\numberwithin{lemme}{section}
\numberwithin{prop}{section}
\numberwithin{theo}{section}

\newcounter{compteurEx}
\counterwithin*{compteurEx}{section}

\newcounter{compteurRem}
\newenvironment{remarque}
{\refstepcounter{compteurRem}
\textbf{Remark \arabic{section}.\arabic{compteurRem}: }}{ \par}
\counterwithin*{compteurRem}{section}

\renewcommand{\theequation}{\arabic{section}.\arabic{equation}}
\counterwithin*{equation}{section}

\makeatletter

\newcommand{\segment}[2]{\left[ #1,#2 \right]}

\newcommand{\intd}[2]{\left[ #1,#2 \right)}
\newcommand{\intgd}[2]{\left( #1,#2 \right)}
\newcommand{\entier}[2]{\llbracket #1,#2 \rrbracket}

\newcommand\cinfty[1]{\mathscr{C}^{\infty}\left(#1\right)}
\newcommand\fcttest[1]{\mathscr{C}^{\infty}_{0}\left(#1\right)}
\newcommand\lp[2]{\mathbf{L}^{#1}\left(#2\right)}

\newcommand\hp[2]{\mathbf{H}^{#1}\left(#2\right)}

\newcommand\sphere[1]{\mathbb{S}^{#1}}
\newcommand\domaine[1]{\mathscr{D}\left(#1\right)}

\newcommand{\abs}[1]{\left|#1\right|}
\newcommand{\scalaire}[2]{\left\langle #1,#2\right\rangle}
\newcommand{\norme}[1]{\left\|#1\right\|}

\newcommand\tq{,\quad}
\newcommand\vi{,\text{ }}
\newcommand\paren[1]{\left(#1\right)}

\makeatother

\begin{document}

\title{On a quantum Hamiltonian in a unitary magnetic field with axisymmetric potential}

\author{Paul Geniet \thanks{
Institut de Mathématiques de Bordeaux, UMR 5251 du CNRS, Université de Bordeaux 
351 cours de la Libération - F 33 405 TALENCE, 
E-mail adress: paul.geniet@math.u-bordeaux.fr}}
\date{\today}

\maketitle ~

\begin{abstract}
We study a magnetic Schrödinger Hamiltonian, with axisymmetric potential in any dimension. 
The associated magnetic field is unitary and non constant. 
The problem reduces to a 1D family of singular Sturm-Liouville operators on the half-line indexed by a quantum number. 
We study the associated band functions. 
They have finite limits that are the Landau levels. 
These limits play the role of thresholds in the spectrum of the Hamiltonian. 
We provide an asymptotic expansion of the band functions at infinity. 
Each Landau level concerns an infinity of band functions and each energy level is intersected by an infinity of band functions. 
We show that among the band functions that intersect a fixed energy level, the derivative can be arbitrary small. 
We apply this result to prove that even if they are localized in energy away from the thresholds, quantum states possess a bulk component. 
A similar result is also true in classical mechanics.
\end{abstract}
\tableofcontents
\renewcommand{\theequation}{\arabic{equation}}
\section*{Introduction}
\label{part_intro}
\subsection*{General context}
The motion of a spinless quantum particle in $\mathbb{R}^{n}$ is described by the spectral properties of the associated Hamiltonian. When the particle moves in a magnetic field, it is the magnetic Laplacian $\paren{-i\nabla-\mathbf{A}}^{2}$ acting on $\lp{2}{\mathbb{R}^{n}}$, where $\mathbf{A}$ is a magnetic potential.\par
One of the simplest example of a magnetic field is the constant one. 
In the case $n\in\left\{2,3\right\}$, this model has been studied from the beginning of quantum mechanics \cite{Lan77} 
and also more recently for the general case $n\geqslant 2$ \cite{Hel96,Dim01}.\par
The variations of a non constant field can induce transport properties for the particle. 
In this context, we focus on magnetic fields that are translationally invariant along one direction. 
For such fields, the Hamiltonian has a band structure and 
transport properties in the direction of invariance are linked to the study of band functions (also called dispersion curves) that are the eigenvalues of the fibered operators. 
Moreover, the propagation of the particle in this direction is determined by the derivatives of these band functions that play the role of group velocities \cite{Yaf08,Exn99}. \par 
In the case $n=2$, one of the studied models of this class is the Iwatsuka model \cite{Iwa85,Man97}. 
For $n=3$, similar models are the planar translationally invariant magnetic fields \cite{Yaf08,Rai08}. 
Let $\paren{r,\theta,z}$ denote the cylindrical coordinates of $\mathbb{R}^{3}$. 
The potential takes the form $\mathbf{A}\paren{r,\theta,z}=\paren{0,0,a\paren{r}}$, where $a:\mathbb{R}\rightarrow\mathbb{R}$ is the intensity of the potential. 
The associated magnetic field is therefore given by 
\begin{equation}
\label{forme_champ_mgt}
\mathbf{B}\paren{r,\theta,z}=a'\paren{r}\paren{-\sin\paren{\theta},\cos\paren{\theta},0}.
\end{equation}
Thus this field is planar and its norm is $\norme{\mathbf{B}\paren{r,\theta,z}}=\abs{a'\paren{r}}$. 
Moreover the associated field lines are circles contained in planes $\left\{z=\text{cst}\right\}$ with center on the invariant axis (see Figure \ref{figure_champ_magnetic}).\par
\begin{figure}[!htb]
\begin{center}
\begin{tabular}{c}
\includegraphics[width=0.3\columnwidth]{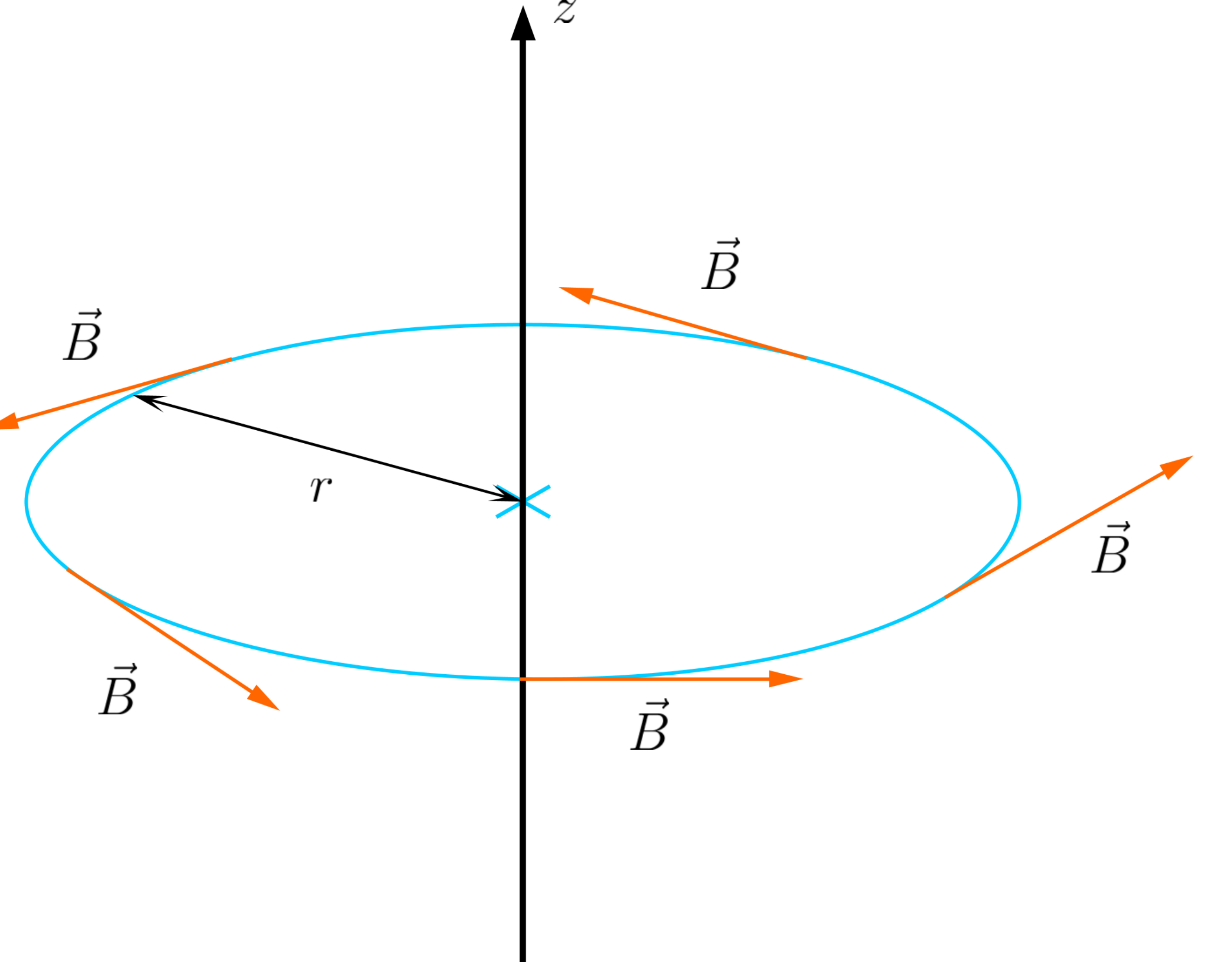}
\end{tabular}
\end{center}
\caption{Schematic of translationally invariant magnetic fields.}
\label{figure_champ_magnetic}
\end{figure}

In view of the form of the magnetic field \eqref{forme_champ_mgt}, two specific cases are relevant. 
The first model consists of a magnetic field generated by an infinite rectilinear wire bearing a constant current \cite{Yaf03,Bru15}. 
If we assume that the wire coincides with the $Oz$ axis, then the Biot \& Savard law states that the generated magnetic fields writes as the field \eqref{forme_champ_mgt}
for the intensity $a\paren{r}=\ln\paren{r}$. 
Here all the band functions are decreasing from $+\infty$ to $0$. 
Hence the spectrum of $H$ is $\sigma\paren{H}=\mathbb{R}_{+}$. 
The band functions tend exponentially to $0$ as the momentum in the $z$-direction tends to infinity and it provides a reaction of the ground state energy of $H$ under an electric perturbation \cite{Bru15}. 
Moreover the particle has a preferable direction of propagation along the $Oz$ axis \cite{Yaf03}.\par
It is also natural to consider the case of a unitary magnetic field. 
For the field \eqref{forme_champ_mgt}, it corresponds to the intensity $a\paren{r}=r$. 
In this case the band functions tend to finite limits that are the Landau levels as the momentum in the $z$-direction tends to infinity \cite[Proposition 3.6]{Yaf08}. 
Therefore the bottom of the spectrum of $H$ is positive. An approximated value has been calculated and used to 
compare the energy on a wedge in a magnetic model and the one coming from the regular part of the wedge \cite{Pop12,Pop15}.\par
In this article we continue to study this magnetic field in the case $a\paren{r}=r$ and we generalize the framework to any dimension $n\geqslant 3$. 
In particular we will show that the derivatives of the band functions possess a new type of behavior.\par
\subsection*{Spectral decomposition of the Hamiltonian and description of the model}
For every $x\in\mathbb{R}^{n}$, we set $r:=\norme{\paren{x_{1},\cdots,x_{n-1}}}_{2}$ 
and we define the magnetic potential $\mathbf{A}$ by
\begin{equation}
\label{forme_potentiel}
\mathbf{A}\paren{x_{1},\cdots,x_{n}}:=\paren{0,\cdots,0,r}.
\end{equation}
We define the Hamiltonian as the following operator, self-adjoint in $\lp{2}{\mathbb{R}^{n}}$: 
\begin{equation}
\label{operateur_H}
H:=\paren{-i\nabla - \mathbf{A}}^{2}.
\end{equation}\par
In order to define the magnetic field we consider, we identify this potential with the $1$-differential form $r\text{d}x_{n}$. 
We define the magnetic field $\mathbf{B}$ as $\mathbf{B}=\text{d}\mathbf{A}$. 
We calculate $B_{j,k}=\paren{\delta_{j,n}-\delta_{n,k}}x_{j}r^{-1}$, $\paren{i,j}\in\entier{1}{n}^{2}$. 
Therefore $\mathbf{B}$ is unitary since $2^{-1}\text{Tr}\paren{\mathbf{B}^{\ast}\mathbf{B}}=2^{-1}\text{Tr}^{+}\paren{\mathbf{B}}=1$ \cite[Section 1]{Hel96}.\par
After a partial Fourier transform in the $x_{n}$ variable, $H$ is unitarily equivalent to the direct integral in $\lp{2}{\mathbb{R}_{\xi};\lp{2}{\mathbb{R}^{n-1}}}$ 
of the family of operators $H\paren{\xi}$, self-adjoint in $\lp{2}{\mathbb{R}^{n-1}}$ and defined by
\begin{equation}
\label{defi_h_xi}
H\paren{\xi}:=-\Delta_{\mathbb{R}^{n-1}}+\paren{r-\xi}^{2}.
\end{equation}
Moreover as we will see in Section \ref{sect_oper_hami} for any frequency $\xi\in\mathbb{R}$, $H\paren{\xi}$ reduces to the orthogonal sum over 
$m\in\mathbb{Z}_{+}$ (called the magnetic quantum numbers) of operators $H_{m}\paren{\xi}$ self-adjoint in $\lp{2}{\mathbb{R}_{+};r^{n-1}dr}$ and defined by
\begin{equation*}
H_{m}\paren{\xi}:=-\frac{1}{r^{n-2}}\partial_{r}\paren{r^{n-2}\partial_{r}}+\frac{m\paren{m+n-3}}{r^{2}}+\paren{r-\xi}^{2}.
\end{equation*} \par
The spectrum of each $H_{m}\paren{\xi}$ is discrete (see Section \ref{sect_spe_ana}). 
Let $\lambda_{m,p}\paren{\xi}$, $p\in\mathbb{N}$ be the increasing sequence of its eigenvalues. 
The $\lambda_{m,p}$ are the band functions (also called dispersion curves).\par
We say that an operator $A$ is fibered \cite[Section XIII.16]{Ree78d} if it can be written as
\begin{equation*}
A:=\int\limits_{M}^{\oplus}{A\paren{\xi}d\xi},
\end{equation*}
with $\paren{M,d\xi}$ a $\sigma$-finite measure space. 
An important class of fibered operators is the one of analytically fibered operators introduced in \cite{Ger98a}. 
In this framework, $M$ is a real analytic manifold and some energy levels are particularly relevant \cite[Theorem 3.1 and Section 3]{Ger98a}. They form a discrete set and are referred as {\sffamily thresholds} \cite[Definition 3.9]{Ger98a}. 
Moreover away from them, some spectral results are rather standard. 
For example a limiting absorption principle as well as propagation estimates hold \cite[Theorem 3.3]{Ger98a} and it is tied to Mourre estimates. 
For a fibered operator $A$, we define the energy-momentum set $\Sigma$ as
\begin{equation*}
\Sigma:=\left\{
\paren{\lambda,\xi}\in\mathbb{R}\times M\tq\lambda\in\sigma\paren{A\paren{\xi}}
\right\}.
\end{equation*}
One of the necessary conditions for the operator $A$ to be analytically fibered in this sense is that 
the projection $\pi:\Sigma\rightarrow\mathbb{R}$ defined as $\pi\paren{\paren{\lambda,\xi}}=\lambda$ is proper. Finally, notice that if $M$ is a 1-dimensional manifold, then these thresholds correspond to the critical values of the band functions and can be referred to as attained thresholds \cite{Geu97,Hel01,Soc01,Bri09}.\par
Other examples of fibered magnetic models can be found in the literature, in dimension 2 \cite{Iwa85}, on the half-plane \cite{Bru14} or in dimension 3 \cite{Yaf08}. 
In these models, the considered Hamiltonian is also fibered along $\mathbb{R}$ and the band functions that are functions of $\xi\in\mathbb{R}$ tend to finite limits as $\xi\to+\infty$. 
The sets of frequencies associated with the energy levels concentrated in the neighborhood of these limits are unbounded. 
Hence the previous projection, $\pi$, is not proper. 
So these magnetic models are not contained in the class of analytically fibered operators that we described above. 
Nevertheless thresholds can still be defined as the limits of the band functions as $\xi\to+\infty$.\par
The model described in this article remains in this case. 
Indeed it is already known that the band functions tend to the Landau levels $E_{p}$ as $\xi\to+\infty$ \cite[Proposition 3.6]{Yaf08}. 
Our first goal is to precise the convergence of the band functions to these levels. 
To that aim we provide an asymptotic expansion for $\lambda_{m,p}\paren{\xi}$ as $\xi\to+\infty$ (see Theorem \ref{theo_asym_dvp}). 
The method used to prove this theorem is inspired by the method of quasi-modes \cite{Dim99} that has already been used in the proof of similar result \cite{Bru15,His16}.\par
For the previous magnetic models, some studies of classical spectral problems 
already exist \cite{Man97,Deb99,His15,His16,Pop16}. 
Our model contains one additional challenge. 
Actually for the Iwatsuka model and for the half-plane model, the thresholds are the limits at infinity of the band functions. 
Moreover, these band functions do not accumulate at any of these thresholds. 
On the contrary, in this article, each threshold $E_{p}$ is the limit of all the band functions 
$\lambda_{m,p}$ for $m\geqslant 0$ at infinity. 
Therefore any interval of energy $I\subset\sigma\paren{H}$ is intersected by an infinity of band functions (see equation \eqref{defi_xi_m}) and 
the set of frequencies $\left\{\lambda_{m,p}^{-1}\paren{I}\vi m\geqslant 0\vi p\in\mathbb{N}\right\}$ associated with $I$ (even if $\overline{I}$ is away from the Landau levels) is unbounded (see Proposition \ref{loc_fr}). 
Furthermore we will prove in Theorem \ref{theo_comp_asym_der} that even if $\overline{I}$ is away from the Landau levels, the supremum $\sup\paren{\lambda_{m,p}'\paren{\xi}\vi\lambda_{m,p}\paren{\xi}\in I}$ tends to $0$ as $m\to+\infty$. 
Therefore it is not clear at first sight that the Mourre estimates used in the case of the analytically fibered operators still hold. Indeed these estimates make use of the fact that away from the thresholds, the derivatives of the band functions are bounded from below by a positive constant \cite[formulas (3.3) to (3.5)]{Ger98a}. 
The proof of Theorem \ref{theo_comp_asym_der} uses a convenient formula for the derivative $\lambda_{m,p}'$ (see Proposition \ref{der_fin_fct_band}) which links this derivative to the normalized eigenfunctions of the operator $H_{m}\paren{\xi}$. 
This proof also uses the exponential decay of these eigenfunctions that is uniform with respect to $m$ and relies on Agmon estimates.\par
These properties have consequences for the transport properties associated with the magnetic field that we consider: define a position operator in the $x_{n}$-direction as the multiplier by the coordinate $x_{n}$. Moreover the time evolution of a quantum state $\varphi$ is given by the Schrödinger equation 
\begin{equation}
\label{eq_schrodinger}
i\partial_{t}\varphi = H\varphi
\end{equation} 
and therefore by the evolution group $e^{-itH}$. 
Combine this with the definition of $x_{n}$. We see with the identity \eqref{eq_just_time_evol_pos} that the position in the $x_{n}$-direction at time $t$ is given by the operator
\begin{equation*}
x_{n}\paren{t}:=e^{itH}x_{n}e^{-itH}.
\end{equation*}  
Define the velocity in the $x_{n}$-direction operator as the time derivative of $x_{n}\paren{t}$. 
This velocity operator has been studied for the Iwatsuka model \cite{Man97} or the 3D model \cite{Yaf08}. 
Let $J$ be the current operator defined as 
\begin{equation*}
\label{defi_vel_ope_intro}
J:=-i\left[H,x_{n}\right]
\end{equation*}
and define the current carried by a state $\varphi$ as $\scalaire{J\varphi}{\varphi}$ \cite{Ens83}. 
Note that (see formula \eqref{eq_der_pos_op}) the velocity in the $x_{n}$-direction is linked to $J$ as follow:
\begin{equation}
\partial_{t}x_{n}\paren{t}=-e^{itH}Je^{-itH}.
\end{equation}
Hence, if $J$ is bounded from below, then $\partial_{t} x_{n}\paren{t}$ is bounded from below.\par
Now let's see how the velocity operator is captured in similar magnetic models and how it is connected to the derivatives of the band functions. 
For the Iwatsuka model (resp. 3D model), the existence of an asymptotic velocity in the $y$-direction (resp. $z$-direction) as $t\to\infty$ has been proven \cite[Theorem 4.2]{Man97}, \cite[Theorem 5.1]{Yaf08}. 
Moreover in both case, the asymptotic velocity is constructed thank to estimates on the derivatives of the band functions \cite[Formula (4.2)]{Man97}, \cite[Formula 5.4]{Yaf08}.\par
For the model on the half-plane, the current operator has been studied \cite{His16}. 
The study distinguishes between two types of behavior: the edge states that carry a non zero current and their counterpart, the bulk states that carry an arbitrarily small one \cite{Hal82,Akk98}, \cite[Section 7]{Hor02}. 
One of the key argument for this study is the decomposition of the current operator thank to the derivatives of the band functions \cite[formula (1.10)]{His16}. In this framework any energy interval $I$ away from the thresholds is intersected by a finit number of band functions. Moreover the derivative of each band function is bounded from below by a positive constant on $I$. Hence the current operator is bounded from below on $I$. Therefore any quantum state localized in energy on $I$ carries a non trivial current \cite{Deb99,Fro00,His08}.
On the counterpart, if there is a threshold in $I$, then there is a band function that intersect $I$ with a arbitrarily small derivative. Hence one can see that the current operator is not bounded from below on $I$ \cite[Section 4]{His08}.\par
In section \ref{sec_ope_cour} we study the current operator associated with the operator \eqref{operateur_H}. 
First we will show that, the current operator is still linked to the multiplier by the family of the derivatives of the band functions (formula \eqref{app_fibr_cour}). 
So in Theorem \ref{theo_trans_courr}, we will apply Theorem \ref{theo_comp_asym_der} that states that for any energy interval $I$ (even if $I$ does not contain a Landau level), the family of the derivatives of the band functions that cross $I$ is not bounded from below on $I$ to see that the current operator is not bounded from below on $I$ either.\par
Finally, as a conclusion, according to Theorem \ref{theo_trans_courr}, the definition of ``thresholds'' as the Landau levels seems not to be relevant in this article: in the case of the model considered here, any quantum state, even localized in energy away from the Landau levels possesses a component with small current (see Theorem \ref{theo_trans_courr} and remark \ref{rem_faibl_courr}). We still denote it a bulk component by analogy with the previous model.\par
In classical mechanics, such a magnetic field also induces transport properties. 
Indeed a charged particle follows the Newton law $m\ddot{\mathbf{x}}=q\dot{\mathbf{x}}\wedge\mathbf{B}$. 
This equation can be integrated \cite[Section 4]{Yaf03} and we plotted the classical trajectories (Figure \ref{figure_traj_class}) in the case $a\paren{r}=r$. 
We can observe that the particle propagates in the $Oz$ direction and one can show that it has an effective velocity $v_{z}$ in this direction: there is a constant $v_{z}$ such that $z\paren{t}=v_{z}t+O\paren{1}$ \cite[Theorem 4.2]{Yaf03}. 
Denote by $\paren{r\paren{t},\theta\paren{t},z\paren{t}}$ the cylindrical coordinates of the particle at time $t$. One can see that $r$ is a periodic function of the time \cite[Formula (4.18)]{Yaf03}. Let $T$ be its period. 
Furthermore, denote by $\sigma:=r^{2}\dot{\theta}$ the areal velocity of the particle that is a constant fixed by the initial conditions \cite[Formula (4.13)]{Yaf03}. 
We deduce the following value for $v_{z}$ \cite[Formula (4.22)]{Yaf03}: 
\begin{equation}
\label{eq_al_v_z}
v_{z}=\frac{\sigma^{2}}{T}\int_{0}^{T}{\frac{dt}{r\paren{t}^{3}}}.
\end{equation}
Let $E$ be the total energy of the particle. Note that $E$ does not depend on time \cite[Formula (4.3)]{Yaf03}. 
Moreover one can see that $r^{2}\dot{\theta}^{2}\leqslant E$ \cite[Formula (4.12)]{Yaf03}. 
Combine it with the definition of $\sigma$ and with the relation \eqref{eq_al_v_z}. 
We get the estimate $\abs{v_{z}}\leqslant E^{3/2}\abs{\sigma}^{-1}$. 
In addition for $\paren{E,\sigma}\in\mathbb{R}_{+}\times \mathbb{R}$, with $\sigma\neq 0$, 
one can find initial conditions such that $E$ is the energy of the particle and $\sigma$ its areal velocity. 
Therefore one can find initial conditions such that $v_{z}$ is arbitrarily small, namely such that the particle propagates arbitrarily slowly along the $Oz$ axis.

\begin{figure}[!htb]
\begin{subfigure}{0.5\textwidth}
\includegraphics[width=0.8\columnwidth]{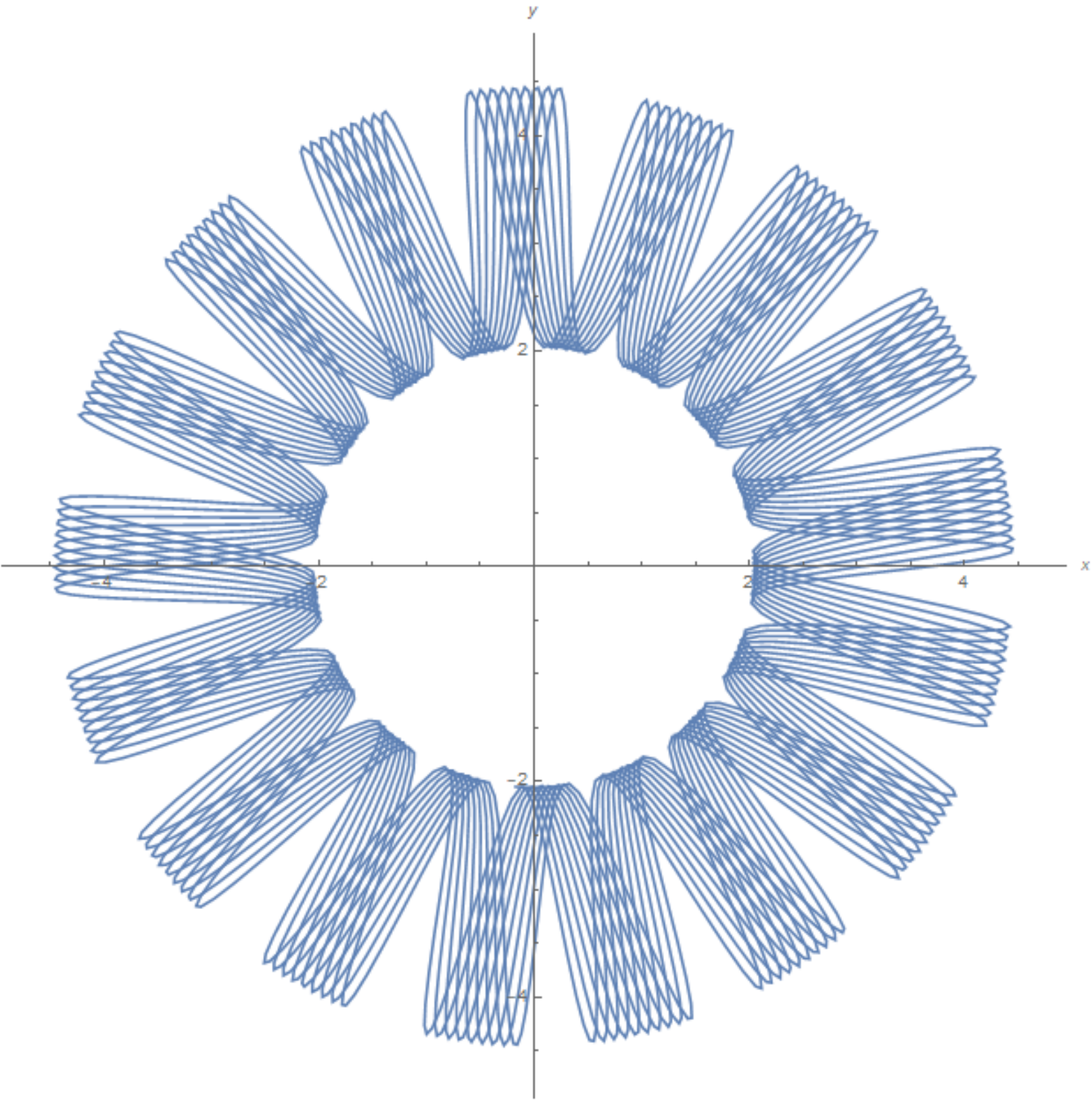}
\caption{Projection in the plan $xOy$.}
\end{subfigure}
\begin{subfigure}{0.5\textwidth}
\includegraphics[width=0.8\columnwidth]{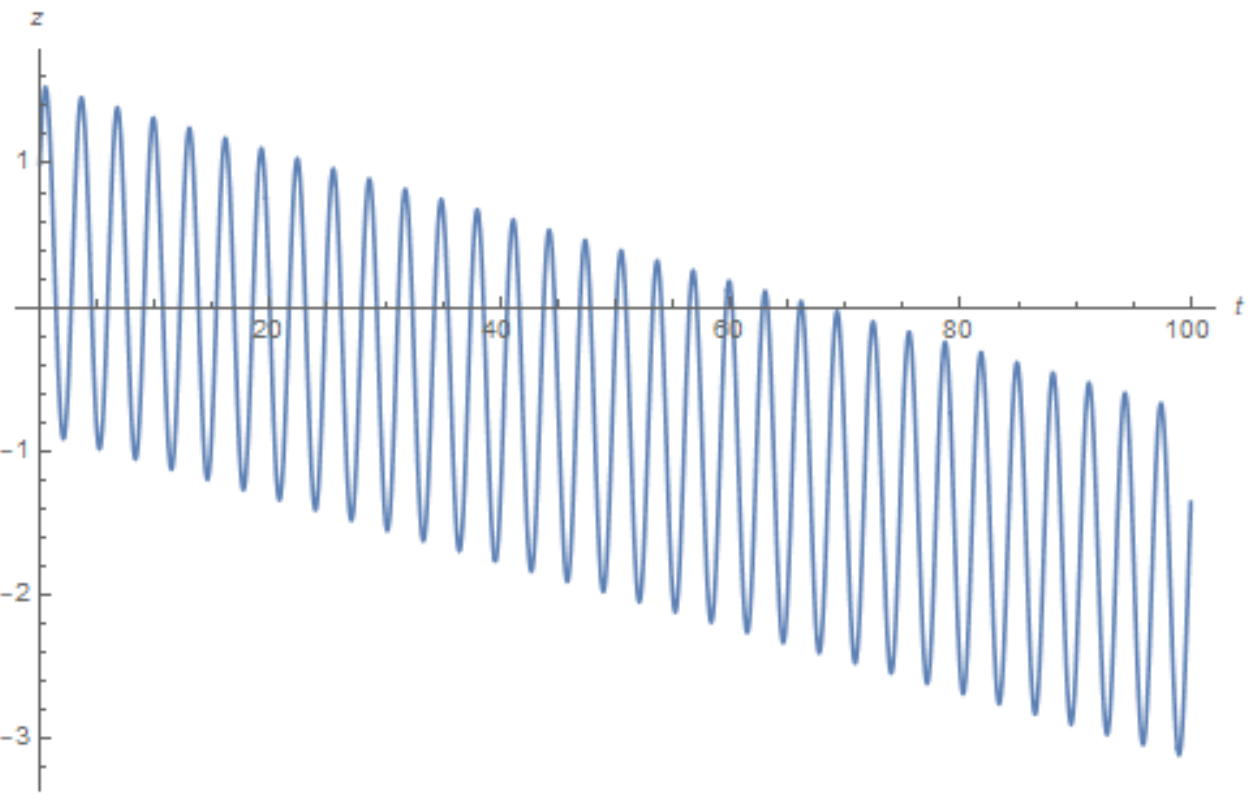}
\caption{Plot of $z$ in function of the time.}
\end{subfigure}
\caption{Trajectories of a charged particle moving in the magnetic field shown by Figure \ref{figure_champ_magnetic}.}
\label{figure_traj_class}
\end{figure}
\subsection*{Organization}
In Section \ref{sect_oper_hami}, the Hamiltonian is reduced to a family of 1D singular Sturm-Liouville operators. 
The band functions are introduced and described in Section \ref{sect_spe_ana}. 
Section \ref{sect_asy_beh} presents the results concerning the asymptotic behaviors of these band functions as $\xi$ and $m$ get large. More precisely, in Subsection \ref{sect_asympt_xi_inf}, we prove Theorem \ref{theo_asym_dvp} that provides an asymptotic expansion of $\lambda_{m,p}\paren{\xi}$ as $\xi$ gets large. 
Subsection \ref{sect_asy_der} presents the asymptotic study of the derivative. 
In particular, Theorem \ref{theo_comp_asym_der} provides the asymptotic behavior of $\lambda_{m,p}'\paren{\xi}$ as $m\to+\infty$ and as $\lambda_{m,p}\paren{\xi}$ is fixed far from the Landau level $E_{p}$. 
In Section \ref{sec_ope_cour}, we analyze the current carried by quantum states that are localized in energy away from the thresholds.\par
\renewcommand{\theequation}{\arabic{section}.\arabic{equation}}
\section{Reduction to one-dimensional Hamiltonians}
\label{sect_oper_hami}
In this section we define precisely the operators that we consider and we explain how $H$ is reduced to 1 dimensional operators.\par
Let $\mathbf{A}:\mathbb{R}^{n}\rightarrow\mathbb{R}^{n}$ be the magnetic potential given by definition \eqref{forme_potentiel} 
and let $H$ be the self-adjoint Schrödinger operator \eqref{operateur_H}. 
This operator is defined via its quadratic form
\begin{equation*}
\label{form_quad_h}
q\paren{u}:=\int_{\mathbb{R}^{n}}{\abs{-i\nabla u\paren{x}-A\paren{x}u\paren{x}}^{2}dx}.
\end{equation*}
This form, initially defined on $\fcttest{\mathbb{R}^{n}}$, is semi-bounded from below. 
Thus it admits a Friedrichs extension: $H$. 
Let $q_{\xi}$ be the quadratic form defined by
\begin{equation*}
\label{form_quad_q_xi}
q_{\xi}\paren{u}:=\int_{\mathbb{R}^{n-1}}{\paren{\abs{\paren{\nabla u}\paren{x}}^{2}+\paren{-\xi+\norme{\paren{x_{1},\cdots,x_{n-1}}}_{2}}^{2}\abs{u\paren{x}}^{2}}dx},\quad\xi\in\mathbb{R}.
\end{equation*}
This form, initially defined on $\fcttest{\mathbb{R}^{n-1}}$ and then closed in $\lp{2}{\mathbb{R}^{n-1}}$, is the quadratic form associated with the operator \eqref{defi_h_xi}. 
Denote by $\mathcal{F}$ the Fourier-transform with respect to $x_{n}$, which is defined by
\begin{equation*}
\label{defi_fourier}
\paren{\mathcal{F}u}\paren{\tilde{x},\xi}:=\frac{1}{\sqrt{2\pi}}\int_{\mathbb{R}}{e^{-i\xi x_{n}}u\paren{\tilde{x},x_{n}}dx_{n}},\quad
\paren{\tilde{x},\xi}\in\mathbb{R}^{n}.
\end{equation*}
The forms $q$ and $q_{\xi}$ are related through the relation
\begin{equation*}
q\paren{u}=\int\limits_{\mathbb{R}}{q_{\xi}\paren{\paren{\mathcal{F}\paren{u}}\paren{\xi}}d\xi}.
\end{equation*}
Therefore the operator $H$ is decomposed as follows: 
\begin{equation*}
\label{dec_int_dir}
H = \mathcal{F}^{-1}\paren{\int_{\mathbb{R}}^{\oplus}{H\paren{\xi}d\xi}}\mathcal{F}.
\end{equation*}
We now reduce the problem to a $1$-dimensional one using both the cylindrical symmetry and the following Laplace-Beltrami formula:
\begin{equation*}
\label{form_lap_bel}
\Delta_{\mathbb{R}^{n-1}}=\frac{1}{r^{n-2}}\partial_{r}\paren{r^{n-2}\partial_{r}}+\frac{1}{r^{2}}\Delta_{\sphere{n-2}}.
\end{equation*}
Recall that $-\Delta_{\mathbb{S}^{n-2}}$ is essentially self-adjoint on $\lp{2}{\sphere{n-2}}$ and that its spectrum is discrete. 
Its eigenvalues are $\mu_{m}:=m\paren{m+n-3}$, $m\in\mathbb{Z}_{+}$. 
Denote by $X_{m}$ the corresponding eigenspaces. Remember that $X_{m}$ has a finite dimension: $N_{m}$. 
The spaces $\lp{2}{\mathbb{R}_{+};r^{n-2}dr}\otimes X_{m}$ are invariant under $H\paren{\xi}$. 
In addition, the restrictions of the operator $H\paren{\xi}$ to these spaces are identified with the operators 
\begin{equation*}
\label{def_h_xi}
H_{m}\paren{\xi}:=-\frac{1}{r^{n-2}}\partial_{r}\paren{r^{n-2}\partial_{r}}+\frac{\mu_{m}}{r^{2}}+\paren{r-\xi}^{2}.
\end{equation*}
These operators act on $\lp{2}{r^{n-2}dr}$. They are associated with the bilinear forms
\begin{equation}
\label{defi_form_h_m}
h_{m}\paren{u,v}:=\int_{0}^{+\infty}{\paren{u'\paren{r}v'\paren{r}+\frac{\mu_{m}}{r^{2}}u\paren{r}v\paren{r}+\paren{r-\xi}^{2}u\paren{r}v\paren{r}}r^{n-2}dr}.
\end{equation}
Denote by $\Phi$ the angular Fourier transform. The operator $H\paren{\xi}$ is decomposed as: 
\begin{equation*}
H\paren{\xi}=\Phi^{-1}\paren{\bigoplus\limits_{m\in\mathbb{N}}{H_{m}\paren{\xi}}}\Phi.
\end{equation*}\par
Finally, it is more convenient to consider operators acting on the Hilbert space $\lp{2}{\mathbb{R}_{+}}$. 
To proceed we use the isometry $\phi\text{ : }\lp{2}{\mathbb{R}_{+};r^{n-2}dr}\rightarrow\lp{2}{\mathbb{R}_{+};dr}$ defined by $\displaystyle\paren{\phi u}\paren{r}=r^{\paren{n-2}/2}u\paren{r}$. 
We define $k_{m}$ as 
\begin{equation}\label{defi_k_n_m}
k_{m}:=\mu_{m}+\frac{n-2}{2}\paren{\frac{n-2}{2}-1}=\frac{\paren{2m+n-3}^{2}-1}{4}
\end{equation} 
and the functions $V_{m}$ as
\begin{equation}
\label{defi_potentiel}
V_{m}\paren{r,\xi}:=\frac{k_{m}}{r^{2}}+\paren{r-\xi}^{2},\quad\paren{r,\xi}\in\mathbb{R}_{+}\backslash\left\{0\right\}\times\mathbb{R}.
\end{equation}\par
So $H_{m}\paren{\xi} = \phi^{-1} L_{m}\paren{\xi}\phi$ where $L_{m}\paren{\xi}$ is defined by
\begin{equation}
\label{defi_L_m_xi}
L_{m}\paren{\xi}:=-\partial_{r}^{2}+V_{m}\paren{r,\xi}.
\end{equation}
This operator acts on $\lp{2}{\mathbb{R}_{+}}$ with domain $\domaine{L_{m}\paren{\xi}}=\phi\paren{\domaine{H_{m}\paren{\xi}}}$. 
It is associated with the quadratic form
\begin{equation}
\label{defi_form_l_m}
l_{m}\paren{u,\xi}:=\int_{0}^{+\infty}{\paren{\abs{u'\paren{r}}^{2}+V_{m}\paren{r,\xi}\abs{u\paren{r}}^{2}}dr}.
\end{equation}
\section{Basics about the eigenpairs of the fiber operator}
\label{sect_spe_ana}
In this section we prove that the dispersion curves are analytic functions, we calculate their derivative and we investigate the behavior of the eigenfunctions at $0$ .\par
\subsection{Behavior of the eigenfunctions at $0$}
First we investigate the behavior of the functions of $\domaine{L_{m}\paren{\xi}}$ at $0$, namely:
\begin{lemme}
\label{lemme_forme_domaine_L_m_xi}
Let $n\geqslant 3$, $m\in\mathbb{Z}_{+}$ and $\xi\in\mathbb{R}$.
\begin{equation}
\label{cond_zero_general_domaine}
\forall\varepsilon>0,\quad\domaine{L_{m}\paren{\xi}}\subset{\left\{
u\in\lp{2}{\mathbb{R}_{+}},\quad
u\underset{r\to 0}{=}o\paren{r^{\frac{1}{2}-\varepsilon}}
\right\}}.
\end{equation}
Moreover 
\begin{equation}
\label{domaine_L_m_xi}
\begin{array}{l}
\displaystyle\text{if }n=3,\quad \domaine{L_{m}\paren{\xi}}\subset \left\{u\in\lp{2}{\mathbb{R}_{+}},\quad u\paren{r}\underset{r\to 0}{=} O\paren{\sqrt{r}}\right\};\\
\displaystyle\text{if }n=4,\quad \domaine{L_{m}\paren{\xi}}\subset \left\{u\in\lp{2}{\mathbb{R}_{+}},\quad u\paren{r}\underset{r\to 0}{=} O\paren{r}\right\}.
\end{array}
\end{equation}
\end{lemme}
\begin{proof}[of \eqref{cond_zero_general_domaine}]
The bilinear form associated with $H_{m}\paren{\xi}$ is given by relation \eqref{defi_form_h_m}. 
For every $u\in\domaine{H_{m}\paren{\xi}}$ and every $v\in\domaine{h_{m}}$, we have $\scalaire{H_{m}\paren{\xi}u}{v}=h_{m}\paren{u,v}$. 
Notice that $\domaine{h_{m}}\subset\hp{1}{\mathbb{R}_{+}}$. We integrate by part the first term of the form $h_{m}$ which yields:
\begin{equation*}
\label{rel_lim_dom}
\lim_{r\to 0}u'\paren{r}v\paren{r}r^{n-2}=0,\quad u\in\domaine{H_{m}\paren{\xi}}\vi v\in\domaine{h_{m}}.
\end{equation*}
We apply this formula to an arbitrary function $u\in\domaine{H_{m}\paren{\xi}}$ and to functions $v_{\varepsilon}\in\cinfty{\mathbb{R}_{+}}\cap\domaine{h_{m}}$ that satisfy for any $\varepsilon>0$
\begin{equation*}
\begin{array}{ll}
v_{\varepsilon}\paren{r}= r^{\frac{3-n}{2}+\varepsilon},& \text{if }r\in\intgd{0}{1};\\
v_{\varepsilon}\paren{r}= 0,& \text{if } r\geqslant 2.\\
\end{array}
\end{equation*}
We deduce that
\begin{equation*}
u'\paren{r}\underset{r\to 0}{=} o\paren{\frac{1}{r^{\frac{n-1}{2}+\varepsilon}}},\quad\varepsilon>0\vi u\in\domaine{H_{m}\paren{\xi}}.
\end{equation*}
Therefore integrating this condition, we deduce that
\begin{equation*}
\begin{array}{l}
u\paren{r}\underset{r\to 0}{=} o\paren{r^{\frac{3-n}{2}-\varepsilon}},\quad\varepsilon>0\vi u\in\domaine{H_{m}\paren{\xi}}.
\end{array}
\end{equation*}
Thus remembering that $\domaine{L_{m}\paren{\xi}}=\phi\paren{\domaine{H_{m}\paren{\xi}}}$, we conclude that relation \eqref{cond_zero_general_domaine} holds.
\end{proof}
\begin{proof}[ of \eqref{domaine_L_m_xi}]
Note that $\domaine{H\paren{\xi}}\subset\hp{2}{\mathbb{R}^{n-1}}$. 
So if $n\in\left\{3,4\right\}$ then awing to a Sobolev embedding, $\hp{2}{\mathbb{R}^{n-1}}\subset\lp{\infty}{\mathbb{R}^{n-1}}$. 
Hence $\domaine{H\paren{\xi}}\subset\lp{\infty}{\mathbb{R}^{n-1}}$. 
Thus if $u\in\domaine{H_{m}\paren{\xi}}$, then $u\paren{r}$ is bounded as $r\to 0$. 
Combine it with the fact that $\domaine{L_{m}\paren{\xi}}=\phi\paren{\domaine{H_{m}\paren{\xi}}}$ and 
it provides the embedding \eqref{domaine_L_m_xi}.\par
\end{proof}
Notice that $V_{m}\paren{r,\xi}\to+\infty$ as $r\to+\infty$. Therefore the operator $L_{m}\paren{\xi}$ has compact resolvent. 
So for every $\xi\in\mathbb{R}$ and for every $m\in\mathbb{Z}_{+}$ 
the spectrum of $L_{m}\paren{\xi}$ is an increasing sequence of positive eigenvalues $\lambda_{m,p}\paren{\xi}$, $p\in\mathbb{N}$. 
We conclude this subsection by proving the following proposition.
\begin{prop}[Behavior of the eigenfunctions at $0$]
\label{prop_cond_bord_zero}
Let $\xi\in\mathbb{R}$, $m\in\mathbb{Z}_{+}$ and $p\in\mathbb{N}$. 
The eigenvalue $\lambda_{m,p}\paren{\xi}$ is non-degenerate. 
Let $u_{m,p}\paren{\cdot,\xi}$ be the normalized eigenfunction associated with it. 
There exists an analytic function $f$ such that $f\paren{0}\neq 0$ and such that in a neighborhood of $0$, 
\begin{equation}
\label{forme_sol_fuchs}
u_{m,p}\paren{r,\xi}=r^{\frac{1+\abs{2m+n-3}}{2}} f\paren{r}.
\end{equation}
\end{prop}
\begin{proof}
First, consider the differential equation
\begin{equation}
\label{edo_fuch}
r^{2}u''\paren{r}+\paren{r^{2}\paren{\lambda_{m,p}\paren{\xi}-\paren{r-\xi}^{2}}-k_{m}}u\paren{r}=0,\quad r>0.
\end{equation}
We look for solutions that admit a series expansion in a neighborhood of $0$. 
By the Frobenius method, if a solution $u$ is given by $u\paren{r}=r^{\nu}f\paren{r}$ where $f$ is an analytic function such that $f\paren{0}\neq 0$, then $\nu$ satisfies the indicial equation 
$$\nu\paren{\nu-1}=k_{m}.$$ 
This equation has $\nu_{\pm}:=\paren{1\pm\paren{2m+n-3}}/2$ as solutions.
Thus the equation \eqref{edo_fuch} admits a solution of the form $u_{+}\paren{r}=r^{\nu_{+}}f\paren{r}$ with $f$ an analytic function such that $f\paren{0}=1$. 
In order to have a basis of solutions for equation \eqref{edo_fuch} 
we look for a solution of the form $u_{-}=hu_{+}$. 
By straightforward calculations we find that $h'\paren{r}=Ku_{+}^{-2}\paren{r}\sim Kr^{-1-\abs{2m+n-3}}$ as $r\to 0$, so 
\begin{itemize}
\item if $\paren{n,m}=\paren{3,0}$, then $u_{-}\paren{r}\underset{r\to 0}{\sim} K\log\paren{r}\sqrt{r}$,
\item in the other cases, $u_{-}\paren{r}\underset{r\to 0}{\sim} Kr^{\nu_{-}}$.
\end{itemize}
Finally, we deduce from Lemma \ref{lemme_forme_domaine_L_m_xi} that in both cases $u_{-}\not\in\domaine{L_{m}\paren{\xi}}$. 
Hence $\ker\paren{L_{m}\paren{\xi}-\lambda_{m,p}\paren{\xi}}=\text{span}\paren{u_{+}}$. 
This concludes the proof since $\lambda_{m,p}\paren{\xi}$ is an eigenvalue of $L_{m}\paren{\xi}$.
\end{proof}
\begin{remarque}
We deduce from this proposition that the embedding \eqref{domaine_L_m_xi} is optimal.
\end{remarque}
According to Proposition \ref{prop_cond_bord_zero}, the eigenvalues $\lambda_{m,p}\paren{\xi}$ are non degenerate. 
Moreover, $L_{m}\paren{\xi}$ is a Kato analytic family \cite[Chapter VII]{Kat66}. 
Therefore it follows from Proposition \ref{prop_cond_bord_zero} that $\lambda_{m,p}$ are real analytic functions that are called band functions. 
\subsection{Derivative of the band functions}
\label{sub_part_calcul_der}
Here we give a formula for the derivative of the band functions.
\begin{prop}
\label{der_fin_fct_band}
Let, for $\paren{\xi,m,p}\in\mathbb{R}\times\mathbb{Z}_{+}\times\mathbb{N}$, $\ K_{m,p}\paren{\xi}:=\lim\limits_{r\to 0}{\dfrac{u_{m,p}\paren{r,\xi}^{2}}{r}}$.
The derivative $\lambda_{m,p}'\paren{\xi}$ is given by: 
\begin{equation*}
\label{forme_der_int}
\lambda_{m,p}'\paren{\xi}=\left\{
\begin{array}{ll}
\displaystyle -\int\limits_{0}^{+\infty}{\frac{1}{r^{2}}\left[
\frac{u_{m,p}\paren{r,\xi}^{2}}{r}-K_{0,p}\right]
dr}&\text{ if } n=3\text{ and } m=0,\\
-\abs{u'_{m,p}\paren{0,\xi}}^{2}&\text{ if } n=4\text{ and } m=0,\\
\displaystyle -2k_{m}\int\limits_{0}^{+\infty}{\frac{\abs{u_{m,p}\paren{r,\xi}}^{2}}{r^{3}}dr}&\text{ in the other cases.}
\end{array}
\right.
\end{equation*}
\end{prop}
\begin{proof}
In the case $n=3$, this proposition has already been proved \cite[Theorem 4.3]{Yaf08}. 
The way to prove it in the general case is the same as in this particular case so we refer to this proof for more details. 
We still present the main ideas of the proof.\par
The Feynman-Hellmann formula \cite{Mou88} yields that
\begin{equation}
\label{app_fey_hel_form}
\lambda_{m,p}'\paren{\xi}=\int_{\mathbb{R}_{+}}{\partial_{\xi}\paren{\paren{r-\xi}^{2}}\abs{u_{m,p}\paren{r,\xi}}^{2}dr}=
-\int_{\mathbb{R}_{+}}{\partial_{r}\paren{\paren{r-\xi}^{2}}\abs{u_{m,p}\paren{r,\xi}}^{2}dr}.
\end{equation} 
We apply integrations by parts to get the result.
We use the super-exponential decay of eigenfunctions $u_{m,p}\paren{\cdot,\xi}$ for handling the non-integral terms corresponding to $r\to+\infty$ \cite{Shn57,Olv97} 
and Proposition \ref{prop_cond_bord_zero} for handling the non-integral term at $r=0$. 
In the particular case $\paren{n,m}= \paren{3,0}$, the result of Proposition \ref{prop_cond_bord_zero} is not sharp enough. 
In order to improve it, we inject the identity \eqref{forme_sol_fuchs} into the following eigenvalue equation:
\begin{equation*}
L_{0}\paren{\xi}u_{0,p}\paren{r,\xi}=\lambda_{0,p}\paren{\xi}u_{0,p}\paren{r,\xi}.
\end{equation*}
Therefore we obtain that $u_{m,p}\paren{r,\xi}^{2}r^{-1}-K_{0,p}=O\paren{r^{2}}$ as $r\to 0$ and we use it for handling non-integral term at $r=0$.
\end{proof}
\subsection{Global behavior of the band functions}
\label{sub_part_comp_band}
The min-max principle implies that
\begin{equation*}
\label{eq_haute_freq}
\lambda_{m,p}\paren{\xi}\underset{\xi\to-\infty}{\sim}\xi^{2}.
\end{equation*}\par
Indeed first note that if $\xi\leqslant 0$, then $L_{m}\paren{\xi}\geqslant \xi^{2}$.
Therefore 
\begin{equation*}
\lambda_{m,p}\paren{\xi}\geqslant \xi^{2},\quad \xi\leqslant 0.
\end{equation*}
On the other hand, we define for $\varepsilon >0$ the operator $G\paren{\varepsilon}$, self-adjoint on $\lp{2}{\mathbb{R}_{+}}$,
\begin{equation*}
G\paren{\varepsilon}:=-\partial_{r}^{2} +\dfrac{k_{m}}{r^{2}}+\paren{1+\dfrac{1}{\varepsilon}}r^{2}.
\end{equation*}
This operator has compact resolvent, therefore its spectrum is discrete. 
Let $\paren{\nu_{q}\paren{\varepsilon}}_{q\in\mathbb{N}}$ be the increasing sequence of its eigenvalues. 
Note that $L_{m}\paren{\xi}\leqslant G\paren{\varepsilon}+\paren{1+\varepsilon}\xi^{2}$. 
Hence, for any $p\in\mathbb{N}$, 
$\lambda_{m,p}\paren{\xi}\leqslant\nu_{p}\paren{\varepsilon}+\paren{1+\varepsilon}\xi^{2}$.
Thus, 
\begin{equation}
\forall\varepsilon>0,\quad\limsup_{\xi\to-\infty}{\frac{\lambda_{m,p}\paren{\xi}}{\xi^{2}}}\leqslant 1+\varepsilon.
\end{equation}\par
From Proposition \ref{der_fin_fct_band} we deduce that if $\paren{n,m}\neq\paren{3,0}$, then 
for every $p\in\mathbb{N}\vi\lambda_{m,p}'$ is negative on $\mathbb{R}$. 
Therefore in this case the band functions are decreasing. 
So these functions admit finite limits at $+\infty$. 
In the case $n=3$ the min-max principle yields that these limits are the Landau levels \cite[Proposition 3.6]{Yaf08}, namely
\begin{equation}
\label{lim_fct_bande}
\lim_{\xi\to+\infty}{\lambda_{m,p}\paren{\xi}}=E_{p}:=2p-1,\quad p\in\mathbb{N}.
\end{equation}
This proof is still valid if $n>3$ and Subsection \ref{sect_asympt_xi_inf} provides an asymptotic expansion of $\lambda_{m,p}\paren{\xi}$ when $\xi$ tends to $+\infty$. 
In the case $n=3$ then $k_{0}=-4^{-1}<0$. Therefore we will deduce from Theorem \ref{theo_asym_dvp} (see remark \ref{rem_calc_ord_1}) that for every $p\in\mathbb{N}\vi\lambda_{0,p}$ admits local minima (the question of the number of minima stays open). 
In the other cases, according to Proposition \ref{der_fin_fct_band}, 
for every $p\in\mathbb{N}\vi\lambda_{m,p}$ is decreasing from $+\infty$ to $E_{p}$.\par
\paragraph{Numerical approximation.}
We use a finite difference method to compute numerical approximations of the band function $\lambda_{m,p}\paren{\xi}$ 
with $n=5$, $m\in\entier{0}{6}$ and $p\in\entier{1}{3}$. 
We compute for $\xi\in\segment{-1}{6}$ on the interval $\segment{0}{20}$ with an artificial Dirichlet boundary condition at $r=20$.\par
On Figure \ref{plot_band_fct}, we have plotted the numerical approximation of $\lambda_{m,p}\paren{\xi}$ 
for $\xi\in\segment{-1}{6}$, $m\in\entier{0}{3}$ and $p\in\entier{1}{3}$. 
According to the theory, $\lambda_{m,p}$ decrease from $+\infty$ to $E_{p}=2p-1$. We also ploted this level. 
Note that different band function may intersect for different values of $m$.\par
Figure \ref{zoom_band_p_1} presents a zoom on the first level: $p=1$ for  $\xi\in\segment{-1}{6}$ and $m\in\entier{0}{6}$.\par
\begin{figure}[!htb]
\begin{center}
\begin{tabular}{c}
\includegraphics[width=\columnwidth]{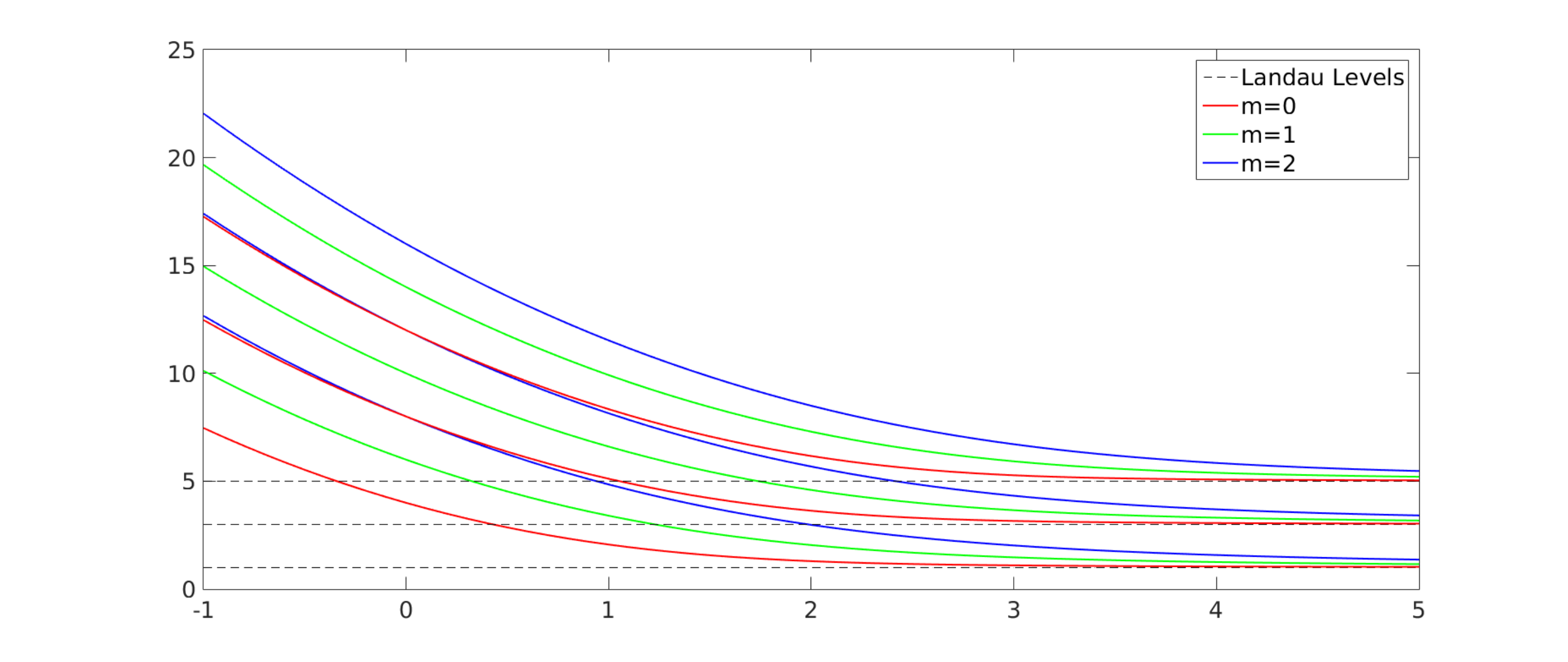}
\end{tabular}
\end{center}
\caption{Plot of the band functions $\lambda_{m,p}\paren{\xi}$ for 
$n=5$, $0\leqslant m\leqslant 3$, 
$1\leqslant p\leqslant 3$ and $\xi\in\left[-1,5\right]$.}\label{plot_band_fct}
\end{figure}
\begin{figure}[!htb]
\begin{center}
\begin{tabular}{c}
\includegraphics[width=\columnwidth]{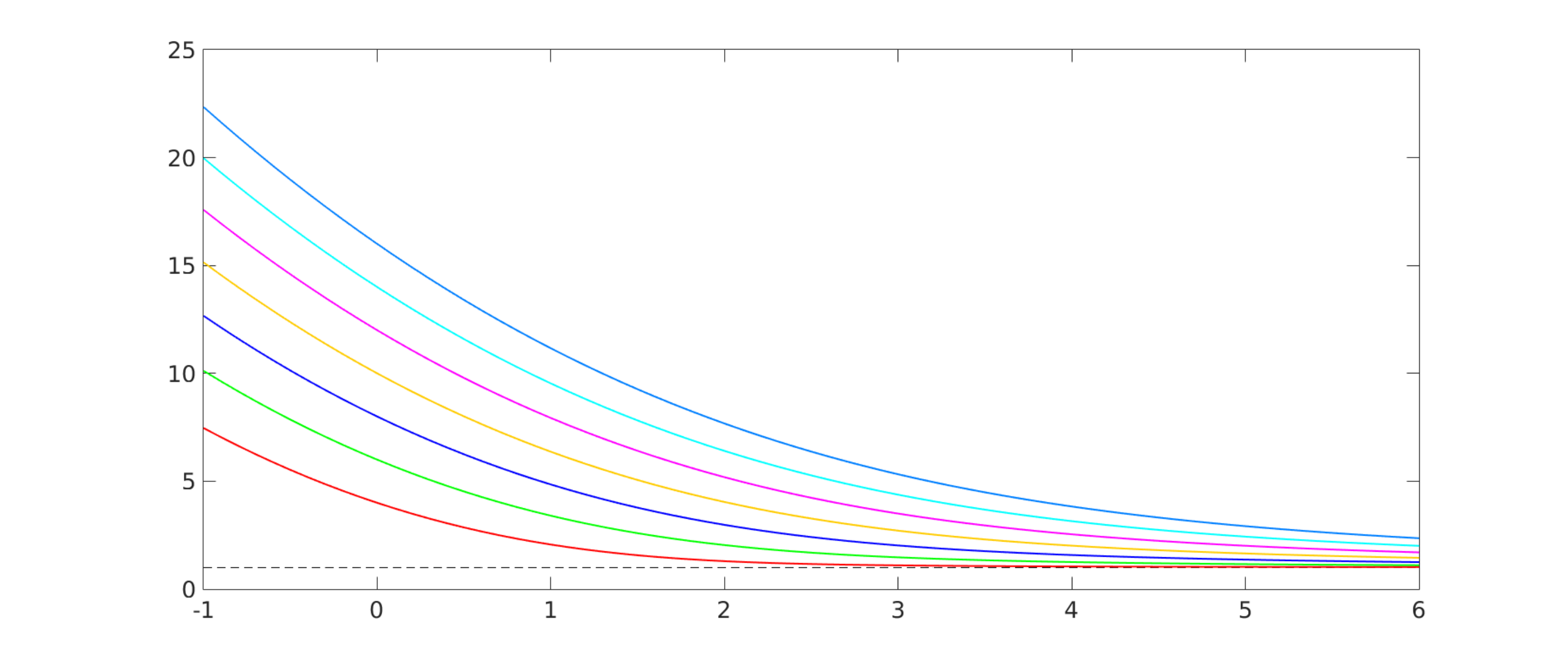}
\end{tabular}
\end{center}
\caption{Plot of the band functions $\lambda_{m,1}\paren{\xi}$ for 
$n=5$, $0\leqslant m\leqslant 6$ and $\xi\in\left[-1,6\right]$.}\label{zoom_band_p_1}
\end{figure}
Graph courtesy of N. Popoff.
\section{Asymptotic behavior of the band functions}
\label{sect_asy_beh}
In this section we provide an asymptotic expansion for the band functions and their derivative. 
First we provide an asymptotic expansion for $\lambda_{m,p}\paren{\xi}$ as $\xi\to+\infty$ with $m$ and $p$ fixed. 
In a second time we estimate the behavior of $\lambda_{m,p}'\paren{\xi}$ as $p$ is fixed and as $m$ and $\xi$ tend to $+\infty$ and 
are related to eachother by the condition $\lambda_{m,p}\paren{\xi}=E$ where $E$ is a constant.\par
\subsection{Near thresholds: high frequency}
\label{sect_asympt_xi_inf}
In this subsection we study the behavior of the spectrum of $H$ near the thresholds. 
Namely we describe the behavior of $\lambda_{m,p}\paren{\xi}$ when $m$ and $p$ are fixed and $\xi\to+\infty$. 
More precisely, this subsection is devoted to the proof of the following theorem.
\begin{theo}[Asymptotic expansion of the band functions]
\label{theo_asym_dvp}
Let $m\in\mathbb{Z}_{+}$ and $p\in\mathbb{N}$. 
There is a sequence of real numbers $\paren{\alpha^{p}_{q}}_{q\in\mathbb{N}}$ such that 
\begin{equation*}
\forall N\geqslant 0\vi\exists C>0\vi\exists\xi_{0}>0\vi\forall\xi\geqslant\xi_{0},\quad
\abs{
\lambda_{m,p}\paren{\xi}
-E_{p}-k_{m}\sum_{q=1}^{N}{\frac{\alpha^{p}_{q}}{\xi^{q}}}
}
\leqslant \dfrac{C}{\xi^{N+1}}.
\end{equation*}
\end{theo}
To prove this theorem we consider the operators $L_{m}\paren{\xi}$ defined by relation \eqref{defi_L_m_xi} and 
we apply the method of the harmonic approximation \cite{Hel88,Dim99} to derive an asymptotic expansion of its eigenvalues.\par
\begin{remarque}
In the case $k_{m}=0$, that is $\paren{n,m}=\paren{4,0}$, Theorem \ref{theo_asym_dvp} states that 
$\lambda_{m,p}\paren{\xi}=E_{p}+O\paren{\xi^{-\infty}}$, as $\xi\to +\infty$. 
In this case, the operator is $-\partial_{r}^{2}+\paren{r-\xi}^{2}$ with Dirichlet boundary condition at $0$. 
This operator has already been studied and we know \cite[Theorem 1.4]{His16} \cite[Section 15.A]{Ivrii} that there are some constant $\gamma_{p}>0$ such that 
\begin{equation*}
\lambda_{0,p}\paren{\xi}\underset{\xi\to+\infty}{=}E_{p}+\gamma_{p}\xi^{2p-1}e^{-\xi^{2}}\paren{1+O\paren{\xi^{-2}} }.
\end{equation*}
So we focus on the proof in the particular case $k_{m}\neq 0$.
\end{remarque}

\begin{remarque}
\label{rem_calc_ord_1}
We compute that $\alpha_{1} = 0$ and $\alpha_{2}=1$. 
Therefore for $N=2$, Theorem \ref{theo_asym_dvp} yields 
\begin{equation*}
\lambda_{m,p}\paren{\xi}=E_{p}+\frac{k_{m}}{\xi^{2}}+O\paren{\frac{1}{\xi^{3}}}.
\end{equation*}\par
In the case $n=3$ and $m=0$, $k_{m}=-4^{-1}<0$. 
Therefore for every $p\in\mathbb{N}$, $\lambda_{0,p}\paren{\xi}$ tend to $E_{p}$ from below. 
Hence the $\lambda_{0,p}$ have local minima.\par
\end{remarque}
\paragraph{Canonical transformation and asymptotic expansion of the operator}\ \\
For $\xi\in\mathbb{R}_{+}$ we apply the change of variable $s=r-\xi$. 
It shows that $L_{m}\paren{\xi}$ is unitarily equivalent to the following operator acting on $\lp{2}{-\xi,+\infty}$:
\begin{equation*}
\label{ope_equiv_l_m_xi}
\tilde{L}_{m}\paren{\xi}=-\partial_{s}^{2}+\frac{k_{m}}{\paren{s+\xi}^{2}}+s^{2}=
-\partial_{s}^{2}+s^{2}+\frac{k_{m}}{\xi^{2}}\frac{1}{\paren{1+\frac{s}{\xi}}^{2}}.
\end{equation*}
A Taylor expansion of the potential for large $\xi$ provides
\begin{equation}
\label{taylor_expansion}
\dfrac{k_{m}}{\paren{s+\xi}^{2}}=\frac{k_{m}}{\xi^{2}}\sum_{q=0}^{N}{\paren{q+1}\paren{\frac{-s}{\xi}}^{q}}+R_{N}\paren{s,\xi},\quad N\geqslant 0.
\end{equation}
Estimation on the remainder term $R_{N}\paren{s,\xi}$ will be written later (see equation \eqref{controle_reste}). 
We define a sequence of formal operators by 
\begin{equation*}
\label{defi_ope_int}
\left\{\begin{array}{l}
H_{0}:=-\partial_{s}^{2}+s^{2},\\
A_{1}:=0,\\
\forall q\geqslant 2\vi A_{q}:=\paren{q-1}\paren{-s}^{q-2}.
\end{array}\right.
\end{equation*}
For every $N\in\mathbb{N}$, we set 
\begin{equation}
\label{defi_L_m_N}
\tilde{L}_{m}^{N}\paren{\xi}:=H_{0}+k_{m}\sum\limits_{q=1}^{N}{\frac{A_{q}}{\xi^{q}}}, 
\end{equation} 
with the convention $\tilde{L}_{m}^{0}=H_{0}$. 
We set $R_{-2}\paren{s,\xi}=R_{-1}\paren{s,\xi}:=k_{m}\paren{s+\xi}^{-2}$. 
For every $N\geqslant 0$, the operator $\tilde{L}_{m}\paren{\xi}$ can be formally decomposed into:
\begin{equation*}
\label{decomp_op}
\tilde{L}_{m}\paren{\xi}=\tilde{L}_{m}^{N}\paren{\xi}+R_{N-2}\paren{s,\xi}.
\end{equation*}
First we look for quasi-modes for the formal operator $\tilde{L}_{m}^{N}\paren{\xi}$ acting on $\lp{2}{\mathbb{R}}$. 
This formal procedure provides functions defined on $\mathbb{R}$ and we use a suitable cut-off function in $\intgd{-\xi}{+\infty}$ to derive quasi-modes for $\tilde{L}_{m}\paren{\xi}$.
\paragraph{Calculation of the quasi-modes}\ \\
We look for quasi-eigenpairs $\paren{\lambda^{N}_{m}\paren{\xi},f_{m}^{N}\paren{\cdot,\xi}}$ of $\tilde{L}_{m}^{N}\paren{\xi}$ of the form
$$\displaystyle\paren{\lambda^{N}_{m}\paren{\xi},f_{m}^{N}\paren{\cdot,\xi}}=\paren{\alpha_{0}+k_{m}\sum_{q=1}^{N}{\frac{\alpha_{q}}{\xi^{q}}},
\sum_{q=0}^{N}{\frac{g_{q}}{\xi^{q}}}},$$
where the functions $g_{q}$ are mutually orthogonal in $\lp{2}{\mathbb{R}}$. Note that the functions $g_{q}$ may depend on $k_{m}$. 
We are led to solve the system
\begin{equation}
\label{systeme}
\begin{array}{l}
\paren{H_{0}-\alpha_{0}}g_{0}=0,\\
\paren{H_{0}-\alpha_{0}}g_{q}+k_{m}\sum\limits_{j=1}^{q}{\paren{A_{j}-\alpha_{j}}g_{q-j}}=0,\quad q\in\entier{1}{N}.
\end{array}
 \end{equation}
We solve it by induction:
\begin{itemize}
\item $\mathbf{q=0}$\\
Note that $H_{0}$ is the quantum harmonic oscillator. Hence we choose for $\paren{\alpha_{0},g_{0}}$ a couple $\paren{E_{p},\Psi_{p}}$ for $p\in\mathbb{N}$
where $E_{p}=2p-1$ is a Landau level, and $\Psi_{p}$ is the corresponding normalised Hermite function with the convention that 
$\Psi_{1}\paren{s}=\paren{2\pi}^{-1/4}e^{-t^{2}/2}$. 
So from now on we set $\paren{\alpha_{0},g_{0}}=\paren{\alpha_{0}^{p},g_{0}^{p}}=\paren{E_{p},\Psi_{p}}$ for a certain $p\in\mathbb{N}$, fixed. 
All the quantities considered in what follows may depend on the choice of $p$. 
We simplify the notations with omitting this index.
\item \textbf{Induction}\\
We assume that there exists $q_{0}\in\entier{1}{N}$ such that for every $q\leqslant q_{0}-1$, $\alpha_{q}$ and $g_{q}$ have been constructed.\par
The scalar product of the second equation of the system \eqref{systeme} with $g_{0}$ provides the value of $\alpha_{q_{0}}$:\par
\begin{equation*}
\label{trouve_lambda}
\alpha_{q_{0}}=\scalaire{A_{q_{0}}g_{0}}{g_{0}}+\scalaire{\sum_{q=1}^{q_{0}-1}{\paren{A_{q}-\alpha_{q}}g_{q_{0}-q}}}{g_{0}}.
\end{equation*}
So $\alpha_{q_{0}}$ is known, therefore the Fredholm alternative provides a unique value for $g_{q_{0}}$ such that $\scalaire{g_{q_{0}}}{g_{q}}=0$ for every $q<q_{0}$. 
\end{itemize}
The quasi-modes $f^{N}_{m}\paren{\cdot,\xi}$ can be computed using the Hermite functions. 
The Hermite functions satisfy the following results
\begin{equation*}
\begin{array}{l}
\forall q\geqslant 1\vi\exists P\in\mathbb{R}\left[X\right]\vi\forall s\in\mathbb{R},\quad\Psi_{q}\paren{s}=e^{-\frac{s^{2}}{2}}P\paren{s},\\
\forall q\geqslant 1,\quad s\Psi_{q}\paren{s}=\sqrt{\dfrac{q-1}{2}}\Psi_{q-1}\paren{s}+\sqrt{\dfrac{q}{2}}\Psi_{q+1}\paren{s}.\\
\end{array}
\end{equation*}
Combining them with the system \eqref{systeme} we infer that for every $N\geqslant 0$, there exist polynomial functions $P_{0},\cdots,P_{N}$ such that
\begin{equation}
\label{forme f_N_xi}
f_{m}^{N}\paren{s,\xi}=e^{-\frac{s^{2}}{2}}\sum_{q=0}^{N}{\frac{P_{q}\paren{s}}{\xi^{q}}}, \quad\xi>0\vi s\in\mathbb{R} .
\end{equation}
\paragraph{Evaluation of the quasi-mode}\ \\
Previously we have obtained quasi-eigenpairs $\paren{\lambda_{m}^{N}\paren{\xi},f_{m}^{N}\paren{\cdot,\xi}}$ for $\tilde{L}_{m}^{N}\paren{\xi}$.
The functions $f_{m}^{N}$ are defined on $\mathbb{R}$. 
We now use a suitable cut-off function to get quasi-modes $u^{N}_{m}\paren{\cdot,\xi}$ for $\tilde{L}_{m}\paren{\xi}$.\par
Let $\chi\in\fcttest{\mathbb{R};\left[0,1\right]}$ such that
\begin{equation*}
\chi\paren{x}=\left\{\begin{array}{l}
1\quad \text{if }\abs{x}\leqslant 1/2;\\
0\quad \text{if }\abs{x}\geqslant 1.
\end{array}\right.
\end{equation*}\par
For $\xi\in\mathbb{R}_{+}$, we define the cut-off function $\chi_{\xi}$ on $\mathbb{R}$ by 
\begin{equation}
\label{cuttoff}
\chi_{\xi}\paren{t}:=\chi\paren{\frac{2t}{\xi}},\quad t\in\mathbb{R}.
\end{equation}
Note that this function is supported in $\intgd{-\xi/2}{\xi/2}$ and is equal to $1$ on $\intgd{-\xi/4}{\xi/4}$. 
Let, for $N\geqslant 0$, $u_{m}^{N}$ be defined by
\begin{equation}
\label{quasimod}
u_{m}^{N}\paren{r,\xi}:=\chi_{\xi}\paren{r}f_{m}^{N}\paren{r,\xi},\quad \xi>0\vi r\in\mathbb{R}.
\end{equation}\par
Since $\text{supp}\paren{u_{m}^{N}\paren{\cdot,\xi}}\subset \text{supp}\paren{\chi_{\xi}}\subset\intgd{-\xi/2}{\xi/2}$, 
$u_{m}^{N}$ can be used as a quasi-mode for $\tilde{L}_{m}\paren{\xi}$.
\begin{lemme}[Control of the quasi-mode]
\label{lemme_ctrl_quas_mod}
Let $N\in\mathbb{Z}_{+}$. 
Recalling that $m$, $N$ and $p$ are fixed, there is a constant $K\geqslant 0$ such that
\begin{equation*}
\label{eq_pou_app_th_spec}
\exists\xi_{0}>0\vi\forall\xi\geqslant\xi_{0},\quad\norme{\paren{\tilde{L}_{m}\paren{\xi}-\lambda^{N}_{m}\paren{\xi}}u_{m}^{N}\paren{\cdot,\xi}}_{2}\leqslant\dfrac{K}{\xi^{N+1}}
\end{equation*}
\end{lemme}
\begin{proof}
First, observe that
\begin{equation}
\label{eq_norme_decomp_ope}
\begin{array}{c}
\norme{\paren{\tilde{L}_{m}\paren{\xi}-\lambda_{m}^{N}\paren{\xi}} u_{m}^{N}\paren{\cdot,\xi}}_{2}\leqslant 
\norme{\chi_{\xi}\paren{\tilde{L}_{m}^{N}\paren{\xi}-\lambda_{m}^{N}\paren{\xi}}f_{m}^{N}\paren{\cdot,\xi}}_{2}+
\norme{R_{N-2}\paren{\cdot,\xi}u_{m}^{N}\paren{\cdot,\xi}}_{2}\\
+\norme{\left[\tilde{L}_{m}^{N}\paren{\xi},\chi_{\xi}\right]f_{m}^{N}\paren{\cdot,\xi}}_{2}.
\end{array}
\end{equation}
We proceed to control the right hand side term by term:
\begin{itemize}
\item We use the definition of $f_{m}^{N}$ to compute the first term:
$$\paren{\tilde{L}_{m}^{N}\paren{\xi}-\lambda_{m}^{N}\paren{\xi}}f_{m}^{N}\paren{\cdot,\xi}=
\sum_{q=N+1}^{2N}{
\frac{k_{m}}{\xi^{q}}\sum_{i+j=q}{\paren{A_{i}-\alpha_{i}}g_{j}}
}.$$
Thus we deduce that
$$\exists K>0,\quad
\norme{\chi_{\xi}\paren{\tilde{L}_{m}^{N}\paren{\xi}-\lambda^{N}_{m}\paren{\xi}}f_{m}^{N}\paren{\cdot,\xi}}_{2}\leqslant\dfrac{Kk_{m}}{\xi^{N+1}}.$$
Note that $K$ may depend on $m$.
\item
Remind that $R_{N}\paren{s,\xi}$ is defined by relation \eqref{taylor_expansion}, the localization of $\text{supp}\paren{\chi_{\xi}}$ provides the following estimate:
\begin{equation*}
\begin{array}{ll}
\exists C>0\vi\forall s\in \text{supp}\paren{\chi_{\xi}},\quad \abs{R_{N-2}\paren{s,\xi}}\leqslant\dfrac{C s^{N+1}}{\xi^{N+1}},&\text{if }N\geqslant 2;\\
\forall s\in \text{supp}\paren{\chi_{\xi}},\quad \abs{R_{N-2}\paren{s,\xi}}\leqslant 4\dfrac{k_{m}}{\xi^{2}},&\text{if }N\in\left\{0,1\right\}.
\end{array}\end{equation*}
Hence using the exponential decay of Hermite functions, we deduce from the definition of $u_{m}^{N}$ and from relation \eqref{forme f_N_xi} that 
\begin{equation}\label{controle_reste}\begin{array}{ll}
\norme{R_{N-2}\paren{\cdot,\xi}u_{m}^{N}\paren{\cdot,\xi}}_{2}\leqslant\dfrac{K}{\xi^{N+1}}&\text{if }N\geqslant 2,\\
\norme{R_{N-2}\paren{\cdot,\xi}u_{m}^{N}\paren{\cdot,\xi}}_{2}\leqslant 4\dfrac{k_{m}}{\xi^{2}}&\text{if }N\in\left\{0,1\right\}.
\end{array}\end{equation}
\item 
Finally notice that $\left[\tilde{L}_{m}^{N}\paren{\xi},\chi_{\xi}\right]f_{m}^{N}\paren{\cdot,\xi}=
2\chi_{\xi}'\paren{f_{m}^{N}}'\paren{\cdot,\xi}+\chi_{\xi}''f_{m}^{N}\paren{\cdot,\xi}$. 
Moreover, $\chi_{\xi}'$ and $\chi_{\xi}''$ are supported in $\left\{t\in\mathbb{R},\quad\xi/4<\abs{t}<\xi/2\right\}$. 
Therefore we deduce from formula \eqref{forme f_N_xi} that
\begin{equation*}
\label{ctr_ter_1}
\norme{\left[\tilde{L}_{m}^{N}\paren{\xi},\chi_{\xi}\right]f_{m}^{N}\paren{\cdot,\xi}}_{2}=O\paren{\frac{1}{\xi^{\infty}}}.
\end{equation*}
\end{itemize}
\end{proof}

\paragraph{Proof of Theorem \ref{theo_asym_dvp}}\ \\
We deduce from the spectral theorem and from Lemma \ref{lemme_ctrl_quas_mod} that
\begin{equation*}
d\paren{\lambda^{N}_{m}\paren{\xi},\sigma\paren{L_{m}\paren{\xi}}}\norme{u_{m}^{N}\paren{\cdot,\xi}}_{\lp{2}{-\xi,+\infty}}\leqslant
\frac{K}{\xi^{N+1}}.
\end{equation*}\par
Moreover $\norme{u_{m}^{N}\paren{\cdot,\xi}}_{\lp{2}{-\xi,+\infty}}=
\norme{f_{m}^{N}\paren{\cdot,\xi}}_{\lp{2}{\mathbb{R}}}+O\paren{\xi^{-\infty}}$ 
and $\norme{f_{m}^{N}\paren{\cdot,\xi}}_{\lp{2}{\mathbb{R}}}=1+O\paren{\xi^{-2}}$. Therefore
\begin{equation}
\label{lim_norme_quasi_mod}
\lim_{\xi\to+\infty}{\norme{u_{m}^{N}\paren{\cdot,\xi}}_{\lp{2}{-\xi,+\infty}}}=1
\end{equation}
Hence for $\xi$ large enough
\begin{equation*}
d\paren{\lambda^{N}_{m}\paren{\xi},\sigma\paren{L_{m}\paren{\xi}}}\leqslant
\frac{\tilde{K}}{\xi^{N+1}}.
\end{equation*}\par
Finally we observe that $\lambda^{N}_{m}\paren{\xi}\to E_{p}$, as $\xi\to+\infty$. 
We combine it with the identity \eqref{lim_fct_bande} that provides the statement of the theorem.\par
\subsection{Near other energy levels: high frequency and high angular momentum}
\label{sect_asy_der}
We are now interested in the behavior of the spectrum of $H$ near other energy levels. 
First if $\xi$ is fixed, then $\lambda_{m,p}\paren{\xi}$ tends to $+\infty$ as $m\to+\infty$. 
In a second time we study the behavior of the band functions when $m$ and $\xi$ tends to $+\infty$ together. 
More precisely we fix an integer $p$ and an energy level $E>E_{p}$ and we study the behavior of $\lambda_{m,p}'\paren{\xi}$ when $\lambda_{m,p}\paren{\xi}=E$.\par
Remember that the quadratic form defined by equation \eqref{defi_form_l_m} is associated to $L_{m}\paren{\xi}$ 
and that $u_{m,p}\paren{\cdot,\xi}$ denotes the normalized eigenfunction of $L_{m}\paren{\xi}$ associated with the eigenvalue $\lambda_{m,p}\paren{\xi}$. 
Therefore, 
\begin{equation}
\label{fct_band_quad_form}
\lambda_{m,p}\paren{\xi}=\int_{0}^{+\infty}{\abs{u_{m,p}'\paren{r,\xi}}^{2}+V_{m}\paren{r,\xi}u_{m,p}\paren{r,\xi}^{2}dr}.
\end{equation}\par
Moreover $k_{m}$ (resp. $V_{m}$) is defined by relation \eqref{defi_k_n_m} (resp. relation \eqref{defi_potentiel}).
Hence for every $m>0\vi k_{m}\geqslant 0$. 
Therefore the following useful estimates are valid for every $m>0$:
\begin{equation}
\label{maj1}\lambda_{m,p}\paren{\xi}\geqslant k_{m}\int_{0}^{+\infty}{\frac{\abs{u_{m,p}\paren{r,\xi}}^{2}}{r^{2}}dr},\\
\end{equation}
\begin{equation}
\label{maj2}\lambda_{m,p}\paren{\xi}\geqslant \int_{0}^{+\infty}{\paren{r-\xi}^{2}\abs{u_{m,p}\paren{r,\xi}}^{2}dr}.
\end{equation}
\begin{prop}[Limit of the band functions]
\label{theo_lim_lam_m}
For every $p\in\mathbb{N}$ and every $\xi\in\mathbb{R}$,
$$\displaystyle\lim_{m\to+\infty}{\lambda_{m,p}\paren{\xi}}=+\infty.$$
\end{prop}
\begin{proof}
We simplify the notations by omitting the index $p$. According to estimate \eqref{maj1}, 
\begin{equation}\label{maj3}
\lambda_{m}\paren{\xi}\geqslant\frac{k_{m}}{R_{0}^{2}}\int_{0}^{R_{0}}{\abs{u_{m}\paren{r,\xi}}^{2}dr}=
\frac{k_{m}}{R_{0}^{2}}\paren{1-\int_{R_{0}}^{+\infty}{\abs{u_{m}\paren{r,\xi}}^{2}dr}},\quad R_{0}>0.
\end{equation}
Moreover, if $R_{0}\geqslant \xi$, then 
\begin{equation*}\displaystyle\int_{0}^{+\infty}{\paren{r-\xi}^{2}\abs{u_{m}\paren{r,\xi}}^{2}}\geqslant\paren{R_{0}-\xi}^{2}\int_{R_{0}}^{+\infty}{\abs{u_{m}\paren{r,\xi}}^{2}dr}.
\end{equation*}
Therefore, from estimate \eqref{maj2} we deduce that 
\begin{equation}\label{maj4}
\lambda_{m}\paren{\xi}\geqslant\paren{R_{0}-\xi}^{2}\int_{R_{0}}^{+\infty}{\abs{u_{m}\paren{r,\xi}}^{2}dr},\quad R_{0}\geqslant \xi.
\end{equation}
Therefore, combining estimates \eqref{maj3} and \eqref{maj4} we obtain 
\begin{equation*}\lambda_{m}\paren{\xi}\geqslant
\frac{k_{m}}{R_{0}^{2}}\paren{1-\paren{R_{0}-\xi}^{-2}\lambda_{m}\paren{\xi}},\quad R_{0}\geqslant\xi.
\end{equation*}
Hence, recalling that $k_{m}\to+\infty$ as $m\to+\infty$, we deduce that 
\begin{equation*}
\exists M\in\mathbb{N}\vi\forall m\geqslant M,\quad
\lambda_{m}\paren{\xi}\geqslant\frac{k_{m}}{R_{0}^{2}}\paren{1+k_{m}\paren{R_{0}\paren{R_{0}-\xi}}^{-2}}^{-1}\geqslant\frac{\paren{R_{0}-\xi}^{2}}{2}.
\end{equation*}
This is true for all $R_{0}>\xi$. So letting $R_{0}$ tend to $+\infty$ provides the result.
\end{proof}
We now study $\lambda_{m,p}'\paren{\xi}$. 
Remember that for any $m\in\mathbb{N}$ and for any $p\in\mathbb{N}$, $\lambda_{m,p}$ is decreasing from $+\infty$ to $E_{p}$. Therefore 
\begin{equation}
\label{defi_xi_m}
\forall m\in\mathbb{N}\vi\forall p\in\mathbb{N}\vi\forall E>E_{p},\quad\exists ! \xi_{m}\in\mathbb{R},\quad E=\lambda_{m,p}\paren{\xi_{m}}.
\end{equation}
\begin{remarque}
Note that $\xi_{m}$ depends on $E$ and $p$.\par
\end{remarque}
\subsubsection{Preliminary results: some localization properties}
\label{sous_part_loc_fr}
First we look for the behavior of $\xi_{m}$ when $m$ tends to $+\infty$.
\begin{prop}[Control of $\xi_{m}$]
\label{loc_fr}
There exist constants $K_{\pm}>0$ such that as $m$ gets large,
\begin{equation*}
K_{-}\sqrt{k_{m}}\leqslant\xi_{m}\leqslant K_{+}\sqrt{k_{m}}.
\end{equation*}
\end{prop}
To get the lower bound, we use formula \eqref{fct_band_quad_form} and we localize the normalized eigenfunctions 
$u_{m}:=u_{m}\paren{\cdot,\xi_{m}}$ of $L_{m}:=L_{m}\paren{\xi_{m}}$.\par 
\begin{proof}[of the lower bound]
Let $\alpha\in\intd{0}{1}$ and let $\displaystyle R_{m}\paren{\alpha}:=\sqrt{k_{m}\alpha E^{-1}}$. 
We inject $\lambda_{m}\paren{\xi_{m}}=E$ into estimate \eqref{maj1}. 
It yields 
$$\frac{E}{k_{m}}\geqslant \int_{0}^{+\infty}{\frac{\abs{u_{m}\paren{r}}^{2}}{r^{2}}dr}\geqslant
\int_{0}^{R_{m}\paren{\alpha}}{\frac{\abs{u_{m}\paren{r}}^{2}}{r^{2}}dr}\geqslant
\frac{1}{R_{m}\paren{\alpha}^{2}}\int_{0}^{R_{m}\paren{\alpha}}{\abs{u_{m}\paren{r}}^{2}dr}.$$\par
So
\begin{equation}
\label{eq_loc_inf}
\int_{0}^{R_{m}\paren{\alpha}}{\abs{u_{m}\paren{r}}^{2}dr}\leqslant\alpha.
\end{equation}\par
Let $\varepsilon>0$ and let $\displaystyle C\paren{\varepsilon}:=\sqrt{E\varepsilon^{-1}}$. 
We make use of estimate \eqref{maj2} to prove in the same way that,
\begin{equation}
\label{eq_loc_xi}
\int_{\left\{\abs{r-\xi_{m}}\leqslant C\paren{\varepsilon}\right\}}{\abs{u_{m}\paren{r,\xi_{m}}}^{2}dr}\geqslant1-\varepsilon.
\end{equation}\par
We combine these estimates to derive an upper bound for $\xi_{m}$. 
Let $\paren{\varepsilon,\alpha}\in\intgd{0}{1}^{2}$ such that $1-\varepsilon >\alpha$. 
We assume that for some $m\in\mathbb{N}$, 
\begin{equation}
\label{eq_absurde_1}
\intgd{\xi_{m}-C\paren{\varepsilon}}{\xi_{m}+C\paren{\varepsilon}}\subset\intgd{0}{R_{m}\paren{\alpha}}.
\end{equation}
We deduce from estimates \eqref{eq_loc_inf} and \eqref{eq_loc_xi} that 
\begin{equation*}
1-\varepsilon\leqslant\int_{-C\paren{\varepsilon}+\xi_{m}}^{C\paren{\varepsilon}+\xi_{m}}{\abs{u_{m}\paren{r,\xi_{m}}}^{2}dr}\leqslant
\int_{0}^{R_{m}\paren{\alpha}}{\abs{u_{m}\paren{r,\xi_{m}}}^{2}dr}
\leqslant \alpha.
\end{equation*}\par
So hypothesis \eqref{eq_absurde_1} can not hold. 
Moreover according to Proposition \ref{theo_lim_lam_m}, $\xi_{m}\to+\infty$ as $m\to+\infty$. 
Therefore for $m$ large enough $\xi_{m}-C\paren{\varepsilon}\geqslant 0$. Hence,
\begin{equation*}
\exists M>0\vi
\forall m\geqslant M,\quad C\paren{\varepsilon}+\xi_{m}\geqslant R_{m}\paren{\alpha}.
\end{equation*}\par
Thus we deduce the existence of $K_{-}$.
\end{proof}
\begin{proof}[of the upper bound] 
We now examine the second part of Proposition \ref{loc_fr}: we show that $\paren{\xi_{m}k_{m}^{-1/2}}_{m\in\mathbb{N}}$ admits an upper bound.
The key argument is $E\neq E_{p}$. 
Indeed we prove that if $\xi_{m}$ tends too fast to $+\infty$, the limit operator is a quantum harmonic oscillator whose eigenvalues are the Landau levels. 
Let's assume that the sequence $\displaystyle\paren{\xi_{m}k_{m}^{-1/2}}_{m\in\mathbb{N}}$ admits no upper bounds. Up to an extraction, one can assume that 
\begin{equation}\label{hyp_abs_up_bou}
\displaystyle\lim_{m\to+\infty}{\frac{\xi_{m}}{\sqrt{k_{m}}}}=+\infty.
\end{equation}\par
Recall (see Subsection \ref{sect_asympt_xi_inf}) that $H_{0}$ is the quantum harmonic oscillator acting on $\lp{2}{\mathbb{R}}$ 
and that the operator $L_{m}$ is unitarily equivalent to the following operator acting on $\lp{2}{-\xi_{m},+\infty}$:
\begin{equation*}
H_{0}+\paren{\frac{\sqrt{k_{m}}}{\xi_{m}}}^{2}\frac{1}{\paren{1+\frac{s}{\xi_{m}}}^{2}}.
\end{equation*}\par
Let $\paren{E_{q},\Psi_{q}}_{q\in\mathbb{N}}$ be the eigenpairs of $H_{0}$. 
For any $m\in\mathbb{N}$, $q\in\mathbb{N}$, we use the functions $\chi_{\xi_{m}}$ and $u_{m}^{1}\paren{\cdot,\xi_{m}}$ defined by formulas \eqref{cuttoff} and \eqref{quasimod}. 
Note that $\chi_{\xi_{m}}\paren{H_{0}-E_{p}}\Psi_{q}=0$, 
therefore according to estimates \eqref{eq_norme_decomp_ope} and \eqref{controle_reste}, 
\begin{equation*}
\norme{\paren{L_{m}-E_{q}} u^{1}_{m,q}}_{2}\leqslant
\norme{\left[H_{0},\chi_{m}\right]\Psi_{q}}_{2}+4\paren{\frac{\sqrt{k_{m}}}{\xi_{m}}}^{2},\quad q\in\mathbb{N}.
\end{equation*}\par
Moreover, 
\begin{equation*}\norme{\left[H_{0},\chi_{m}\right]\Psi_{q}}_{2}=O\paren{\frac{1}{\xi_{m}^{\infty}}}.
\end{equation*}\par
Recall that $\norme{u_{m}^{1}\paren{\cdot,\xi_{m}}}\to 1$ as $m\to+\infty$ (remember that $\xi_{m}\to+\infty$ as $m\to+\infty$ and see the identity \eqref{lim_norme_quasi_mod}) and that we have assumed that $\sqrt{k_{m}}\xi_{m}^{-1}\to 0$ as $m\to+\infty$. 
We thus conclude from the spectral theorem that $$\lim_{m\to+\infty}{d\paren{\sigma\paren{L_{m}},E_{q}}}=0.$$
It implies that for every $q\in\mathbb{N}\vi d\paren{\left\{\lambda_{m,s}\paren{\xi_{m}},s\geqslant 1\right\},E_{q}} \to 0$ as $m\to +\infty$. 
So for every $q\in\mathbb{N}$, $\lambda_{m,q}\paren{\xi_{m}}\to E_{q}$ as $m\to+\infty$, therefore $E=E_{p}$. 
But we have assumed that $E\neq E_p$, hence the hypothesis \eqref{hyp_abs_up_bou} can not hold and we get the upper-bound.\par
\end{proof}
We now study the potential $V_{m}$, defined by formula \eqref{defi_potentiel}. 
Note that $V_{m}$ is strictly convex and that it verifies $V_{m}\paren{r} \to +\infty$ as $r\to 0$ or $r\to+\infty$. Therefore $V_{m}$ admits an unique minimum on $\mathbb{R}_{+}$, $V_{m}^{\min}$, reached at the single critical point of $V_{m}$: $r_{m}$. In Lemma \ref{loc_extr}, we use Proposition \ref{loc_fr} to localize the quantities $r_{m}$ and $V_{m}^{\min}$.
\begin{lemme}[Localization of extrema]
\label{loc_extr}
There are constants $M\in\mathbb{N}$, $R_{\pm}>0$ and $V_{\pm}>0$ such that for every $m\geqslant M$,
\begin{enumerate}
\item $R_{-}\sqrt{k_{m}}\leqslant r_{m}\leqslant R_{+}\sqrt{k_{m}}$;
\item $V_{-}\leqslant V_{m}^{\min}\leqslant V_{+}$.
\end{enumerate}
Moreover, for any $y>V_{m}^{\min}$, the two solutions $r_{\pm}$ of $V_{m}\paren{r}=y$ satisfy:
\begin{equation*}
\exists K_{\pm}>0\vi\exists M\in\mathbb{N}\vi\forall m\geqslant M,\quad K_{-}\sqrt{k_{m}}\leqslant r_{-}< r_{m}<r_{+}\leqslant K_{+}\sqrt{k_{m}}.
\end{equation*}\par
\end{lemme}
\begin{proof}\ \par
\begin{enumerate}
\item 
First, recall that $r_{m}\in\mathbb{R}_{+}$ is the single critical point of $V_{m}$. 
Therefore, $V_{m}'\paren{r_{m}}=0$ provides
\begin{equation*}\frac{k_{m}}{r_{m}^{3}}=r_{m}-\xi_{m}.
\end{equation*}\par
Since $r_{m}>0$, we deduce that $r_{m}-\xi_{m}>0$.
So according to Proposition \ref{loc_fr}
\begin{equation*}
\exists K_{+}>0\vi\exists M\in\mathbb{N}\vi\forall m\geqslant M,\quad r_{m}> \xi_{m}\geqslant K_{+}\sqrt{k_{m}}.
\end{equation*}
Moreover $\displaystyle0<r_{m}-\xi_{m}\leqslant k_{m}\paren{K_{+}\sqrt{k_{m}}}^{-3}=K_{+}^{-3}k_{m}^{-1/2}$. 
So using $k_{m}\to+\infty$ as $m\to+\infty$, we deduce that 
\begin{equation}
\label{lim_r_xi}
\lim_{m\to+\infty}{r_{m}-\xi_{m}}=0.
\end{equation}
Thus Proposition \ref{loc_fr} provides the result.
\item Recall that $V_{m}^{\min}=V_{m}\paren{r_{m}}$. Hence, according to equation \eqref{lim_r_xi}, 
$V_{m}^{\min}-k_{m}r_{m}^{-2}\to 0$ as $m\to+\infty$. So the first point provides the result.\par
\item 
According to the variations of $V_{m}$, $r_{\pm}$ exists and is solution of the equation $k_{m}r^{-2}+\paren{r-\xi_{m}}^{2}=y$. 
Thus $\paren{r_{\pm}-\xi_{m}}^{2}\leqslant y$ and therefore $\abs{r_{\pm}-\xi_{m}}\leqslant\sqrt{y}$. The result follows from Proposition \ref{loc_fr}.
\end{enumerate}
\end{proof}
\begin{remarque}
We do not know if the limits $\displaystyle\lim_{m\to+\infty}{\frac{\xi_{m}}{\sqrt{k_{m}}}}$ and $\displaystyle\lim_{m\to+\infty}{V_{m}^{\min}}$ exist.
\end{remarque}
\subsubsection{Exponential decay of the eigenfunctions}
\label{sub_sub_sec_agm_est}
Here we introduce some tools to estimate the exponential decay of the eigenfunctions. This is an application of the well-known Agmon estimates for 1D Schrödinger operators with confining potential. In our case we would like to take into account the dependance on $m$. Therefore we are led to perturb the Agmon distance to get some uniform estimates.\par
We define the Agmon distance by: 
\begin{equation*}
d_{m}\paren{r_{1},r_{2}}=\abs{\int_{r_{1}}^{r_{2}}{\sqrt{\paren{V_{m}\paren{r}-E}_{+}}dr}},\quad\paren{r_{1},r_{2}}\in\mathbb{R}_{+}^{2}.
\end{equation*}\par
For $\alpha>3/2$ and for every $m\in\mathbb{N}$, we define $\delta_{m}$ by
\begin{equation*}
\delta_{m}=\delta_{m}\paren{\alpha}:=\frac{\alpha}{\sqrt{k_{m}}}.
\end{equation*}\par
Let $I_{m}$ be defined by
\begin{equation}
\label{defi_I_m}
I_{m}=I_{m}\paren{E}:=\left\{r>0,\quad V_{m}\paren{r}<E\right\}.
\end{equation} 
We recall that we have chosen $E>E_{p}$, therefore $I_{m}\neq\emptyset$. Indeed,
\begin{equation}
\label{prop_e_supp_v_min}
E= l_{m}\paren{u_{m}}>\int_{\mathbb{R_{+}}}{V_{m}\abs{u_{m}}^{2}}\geqslant V_{m}^{\min}\norme{u_{m}}_{2}^{2}=V_{m}^{\min}.
\end{equation}\par
Furthermore, remember that $V_{m}$ is strictly convex and that $V_{m}\paren{r}\to+\infty$ as $r\to 0$. 
Therefore $I_{m}$ is an open bounded interval of $\mathbb{R}_{+}$. 
Recall that the distance between $x\in\mathbb{R}$ and a set $X\subset\mathbb{R}$ is defined as $d_{m}\paren{x,X}:=\inf\paren{d_{m}\paren{x,y}\vi y\in X}$. For every $m\in\mathbb{N}$, we define the function $\Phi_{m}$ on $\mathbb{R}_{+}$ by
\begin{equation}
\label{defi_phi_m_p}
\Phi_{m}=\Phi_{m}\paren{\cdot,\delta_{m}}:=\delta_{m}d_{m}\paren{\cdot,I_{m}}.
\end{equation}\par
The function $\Phi_{m}$ is decreasing on $\intgd{0}{\inf\paren{I_{m}}}$, zero on $I_{m}$ and increasing on $\intgd{\sup\paren{I_{m}}}{+\infty}$. 
Moreover since $I_{m}$ is a bounded interval, we deduce that
\begin{equation*}
\begin{array}{ll}
\displaystyle \Phi_{m}\paren{r}=\delta_{m}\int_{r}^{\inf\paren{I_{m}}}{\sqrt{\paren{V_{m}\paren{r}-E}_{+}}dr},&r<\inf\paren{I_{m}},\\
\Phi_{m}\paren{r}=0,&r\in I_{m},\\
\displaystyle  \Phi_{m}\paren{r}=\delta_{m}\int_{\sup\paren{I_{m}}}^{r}{\sqrt{\paren{V_{m}\paren{r}-E}_{+}}dr},&r>\sup\paren{I_{m}}.
\end{array}
\end{equation*}
Hence, $\Phi_{m}$ satisfies the eikonal equation:
\begin{equation}
\label{eq_eikon}
\abs{\Phi_{m}'\paren{r}}^{2}=\delta_{m}^{2}\paren{V_{m}\paren{r}-E}_{+}.
\end{equation}\par
Notice that $\Phi_{m}$ is a perturbated Agmon distance and that $\delta_{m}\to 0$ as $m\to +\infty$. 
We use this fact to prove the following proposition that provides a uniform control for $e^{\Phi_{m}}u_{m}$.
First of all we use the definition of $\Phi_{m}$ given by equation \eqref{defi_phi_m_p} and a Taylor expansion at $0$ and at $+\infty$ to get the following lemma.
\begin{lemme}
\label{lemme_behavior_phi}
Let $\Phi_{m}$ be the function defined by definition \eqref{defi_phi_m_p}. The behavior of $\Phi_{m}\paren{r}$ as $r\to\partial\mathbb{R}_{+}$ is given by: 
\begin{itemize}
\item $\Phi_{m}\paren{r}=-\alpha\ln\paren{r}+O\paren{1}$ as $r\to 0$;
\item $\Phi_{m}\paren{r}=\dfrac{\delta_{m}r^{2}}{2}+O\paren{r}$ as $r\to +\infty$.
\end{itemize}
\end{lemme}
The following proposition is a well known Agmon estimate result \cite{Agm82}. 
Here we are interested in the uniformity with respect to $m$. 
To that aim we adapt the classical proof of the result \cite{Hel88}.
\begin{prop}
\label{prop_est_agmon}
There exist a constant $K$ and an integer $M$ such that
\begin{equation*}
\forall m\geqslant M,\quad\norme{e^{\Phi_{m}}u_{m}}_{2}\leqslant K.
\end{equation*}
\end{prop}
\begin{proof}
According to Lemma \ref{lemme_behavior_phi}, there is a constant $\beta\in\mathbb{R}$ such that, 
\begin{equation*}
\begin{array}{ll}
e^{2\Phi_{m}\paren{r}}=O\paren{r^{-2\alpha}}& r\to 0;\\
e^{2\Phi_{m}\paren{r}}=O\paren{e^{\frac{\delta_{m}}{2}r^{2}+\beta r}}& r\to +\infty.\\
\end{array}
\end{equation*}
Hence according to Proposition \ref{prop_cond_bord_zero}, 
\begin{equation*}
e^{2\Phi_{m}\paren{r}}u_{m}\paren{r}=O\paren{r^{m+\frac{n}{2}-1-2\alpha}},\quad r\to 0.
\end{equation*}
Therefore for $m$ large enough, $e^{2\Phi_{m}}u_{m}\in\lp{2}{0,1}$.
Moreover according to the Liouville-Green approximation \cite[Chapter 6]{Olv97}, 
\begin{equation*}
u_{m}\paren{r}\sim \paren{V_{m}\paren{r}-E}^{-\frac{1}{4}}e^{-\int{\sqrt{V_{m}\paren{r}-E}dr}},\quad r\to+\infty.
\end{equation*}
Remember that $\int{\sqrt{V_{m}\paren{r}-E}dr}\sim r^{2}/2$ as $r\to+\infty$, we deduce that for $m$ large enough, $e^{\phi_{m}}u_{m}\in\lp{2}{1,+\infty}$, 
therefore $e^{2\phi_{m}}u_{m}\in\lp{2}{\mathbb{R}_{+}}$. 
Moreover an integration by parts yields
\begin{equation*}
\scalaire{-u_{m}''}{e^{2\phi_{m}}u_{m}}=
\int_{\mathbb{R}_{+}}{
\paren{\abs{u_{m}'}^{2}+2\phi_{m}'u_{m}u_{m}'}e^{2\phi_{m}}
}-\left[e^{2\phi_{m}}u_{m}u_{m}'\right]_{0}^{+\infty}.
\end{equation*}
According to what preceds, $\left[e^{2\phi_{m}}u_{m}u_{m}'\right]_{0}^{+\infty}=0$, 
thus $e^{\phi_{m}}u_{m}\in\domaine{h_{m}}$. 
Moreover by combining it with the relations \eqref{fct_band_quad_form} and \eqref{defi_xi_m} we obtain
\begin{equation}
\label{est_base_agmon}
\int_{\mathbb{R}_{+}}{\abs{\paren{e^{\Phi_{m}}u_{m}}'}^{2}}+\int_{\mathbb{R}_{+}}{e^{2\Phi_{m}}\paren{V_{m}-E-\abs{\phi_{m}'}^{2}}\abs{u_{m}}^{2}}=0.
\end{equation}\par
Furthermore, according to estimate \eqref{prop_e_supp_v_min}, $E-V_{m}^{\min}>0$. 
Let for every $m\in\mathbb{N}$, $\varepsilon_{m}:=2^{-1}\paren{E-V_{m}^{\min}}>0$.
Recall that $I_{m}$ is given by definition \eqref{defi_I_m}. We define $I_{\pm}$ as
\begin{equation*}
\begin{array}{l}
I_{-}:=I_{m}\paren{E+\varepsilon_{m}}=\left\{r\in\mathbb{R}_{+}\vi V_{m}\paren{r}<E+\varepsilon_{m}\right\},\\
I_{+}:=\mathbb{R}_{+}\backslash I_{-}=\left\{r\in\mathbb{R}_{+}\vi V_{m}\paren{r}\geqslant E+\varepsilon_{m}\right\}.
\end{array}
\end{equation*}\par
By injecting $\mathbb{R}_{+}=I_{+}\sqcup I_{-}$ into equation \eqref{est_base_agmon}, we prove that
\begin{equation*}
\label{rel_fond_pre_hag}
\begin{array}{c}
\displaystyle \int_{\mathbb{R}_{+}}{\abs{\paren{e^{\phi_{m}}u_{m}}'}^{2}}+\int_{I_{+}}{e^{2\phi_{m}}\paren{V_{m}-E-\abs{\phi_{m}'}^{2}}\abs{u_{m}}^{2}}=\\
\displaystyle-\int_{I_{-}}{e^{2\phi_{m}}\paren{V_{m}-E-\abs{\phi_{m}'}^{2}}\abs{u_{m}}^{2}}\leqslant
\norme{V_{m}-E-\abs{\phi_{m}'}^{2}}_{\lp{\infty}{I_{-}}}
\int_{I_{-}}{e^{2\phi_{m}}\abs{u_{m}}^{2}}..
\end{array}
\end{equation*}\par
Let for $m$ large enough such that $\delta_{m}\leqslant 1$, $\displaystyle C_{m}:=\paren{1-\delta_{m}^{2}}\varepsilon_{m}>0$. 
We combine equation \eqref{eq_eikon} with the definition of $I_{\pm}$ to get
\begin{equation*}
\begin{array}{ll}
V_{m}\paren{r}-E-\abs{\Phi_{m}'\paren{r}}^{2}\geqslant C_{m},&\text{if } r\in I_{+};\\
V_{m}^{min}-E\leqslant V_{m}\paren{r}-E-\abs{\Phi_{m}'\paren{r}}^{2}<C_{m},&\text{if } r\in I_{-}.\\
\end{array}
\end{equation*}\par
So remembering that
$\varepsilon_{m}=\paren{E-V_{m}^{\min}}/2$, we get $\norme{V_{m}-E-\abs{\Phi_{m}'}^{2}}_{\lp{\infty}{I_{-}}}\leqslant E-V_{m}^{\min}$ 
and we deduce that
\begin{equation*}
\label{rel_fond_hag}
\begin{array}{c}
\displaystyle \int_{\mathbb{R}_{+}}{\abs{\paren{e^{\Phi_{m}}u_{m}}'}^{2}}+C_{m}\int_{I_{+}}{e^{2\Phi_{m}}\abs{u_{m}}^{2}}\leqslant
\paren{E-V_{m}^{\min}}\int_{I_{-}}{e^{2\Phi_{m}\paren{r}}\abs{u_{m}}^{2}}.
\end{array}
\end{equation*}\par
We recall that $u_{m}$ is normalized that provides
\begin{equation*}
\displaystyle \int_{\mathbb{R}_{+}}{\abs{\paren{e^{\Phi_{m}}u_{m}}'}^{2}}+
C_{m}\int_{\mathbb{R}_{+}}{e^{2\Phi_{m}}\abs{u_{m}}^{2}}\leqslant
\displaystyle \paren{E-V_{m}^{\min}+C_{m}}\int_{I_{-}}{e^{2\Phi_{m}}\abs{u_{m}}^{2}}\leqslant 
\paren{E-V_{m}^{\min}+C_{m}}e^{2\norme{\Phi_{m}}_{\lp{\infty}{I_{-}}}}.
\end{equation*}\par
Finally we deduce the following estimate
\begin{equation}
\label{ctrl_fct_agmon}
\int_{\mathbb{R}_{+}}{e^{2\Phi_{m}\paren{r}}\abs{u_{m}\paren{r}}^{2}dr}\leqslant
\displaystyle \frac{E-V_{m}^{\min}+C_{m}}{C_{m}}e^{2\norme{\Phi_{m}}_{\lp{\infty}{I_{-}}}}.
\end{equation}\par
The choices of $\delta_{m}$ and $\varepsilon_{m}$ yield $\paren{E-V_{m}^{\min}+C_{m}}C_{m}^{-1}=\paren{3-\delta_{m}^{2}}\paren{1-\delta_{m}^{2}}^{-1}$ . 
Thus $\paren{E-V_{m}^{\min}+C_{m}}C_{m}^{-1} $ is bounded as $m\to +\infty$. 
Moreover the variations of $\Phi_{m}$ ensure that $\norme{\Phi_{m}}_{\lp{\infty}{I_{-}}}=\norme{\Phi_{m}}_{\lp{\infty}{\partial I_{-}}}$. 
Therefore Lemma \ref{loc_extr} provides the following control
\begin{equation*}
\exists K>0\vi \exists M_{0}\in\mathbb{N}\vi\forall m\geqslant M_{0},\quad
\norme{\Phi_{m}}_{\lp{\infty}{I_{-}}}\leqslant K\delta_{m}\sqrt{k_{m}}=K\alpha.
\end{equation*}
We conclude the proof by combining it with estimate \eqref{ctrl_fct_agmon} that provides the expected result.
\end{proof}
\subsubsection{Asymptotic expansion of the derivative}
\label{part_der}
Here we prove the following theorem.
\begin{theo}[Asymptotic behavior of the derivative]
\label{theo_comp_asym_der}
Recall that $\xi_{m}$ is defined by relation \eqref{defi_xi_m}.
There are constants $K_{\pm}>0$ and there exists $M\in\mathbb{N}$ such that
\begin{equation*}
\forall m\geqslant M,\quad \frac{K_{-}}{\sqrt{k_{m}}}\leqslant\abs{\lambda_{m}'\paren{\xi_{m}}}\leqslant \frac{K_{+}}{\sqrt{k_{m}}}.
\end{equation*}
\end{theo}
\begin{remarque}
For further use note that this theorem can be adapted to the case where the energy level is an interval $J$. Namely, 
if $J\subset\mathbb{R}$ denotes an interval such that $\overline{J}\cap\left\{E_{p}\vi p\in\mathbb{N}\right\} =\emptyset$, then
\begin{equation*}
\exists M\in\mathbb{N}\vi\forall m\geqslant M,\forall\xi\in\lambda_{m}^{-1}\paren{J},\quad \frac{K_{-}\paren{J}}{\sqrt{k_{m}}}\leqslant
\abs{\lambda_{m}'\paren{\xi}}
\leqslant \frac{K_{+}\paren{J}}{\sqrt{k_{m}}}.
\end{equation*}
\end{remarque}
\begin{remarque}
If $J$ is on the form $\intgd{E_{p}}{E_{p}+\eta}$, them the combinaison of Theorem \ref{theo_asym_dvp} and of Proposition \ref{loc_fr} states that 
there is a constant $C>0$ such that if $\lambda_{m}\paren{\xi}\in J$ then $\xi\geqslant C\sqrt{k_{m}\eta^{-1}}$. 
Therefore one could prove that 
\begin{equation*}
\exists C>0\vi\exists M\in\mathbb{N}\vi\forall m\geqslant M\vi\forall\xi\in\lambda_{m}^{-1}\paren{J},\quad\abs{\lambda_{m}'\paren{\xi}}\leqslant C\sqrt{\frac{\eta}{k_{m}}}.
\end{equation*}
\end{remarque}
\paragraph{Lower bound}\ \\
\label{sub_part_min_der}
According to Proposition \ref{der_fin_fct_band}, 
\begin{equation*}\displaystyle 
\abs{\lambda_{m}'\paren{\xi}}\geqslant\frac{2k_{m}}{R^{3}}\int_{0}^{R}{\abs{u_{m}\paren{r,\xi}}^{2}dr},\quad\xi\in\mathbb{R}\vi R>0\vi m\geqslant 1.
\end{equation*}\par
Let's combine it with estimate \eqref{maj4} and with $\norme{u_{m}\paren{\cdot,\xi}}=1$. We deduce that
\begin{equation*}
\abs{\lambda_{m}'\paren{\xi}}\geqslant\frac{2k_{m}}{R^{3}}\paren{1-\int_{R}^{+\infty}{\abs{u_{m}\paren{r,\xi}}^{2}dr}}
\geqslant\frac{2k_{m}}{R^{3}}\paren{1-\frac{\lambda_{m}\paren{\xi}}{\paren{R-\xi}^{2}}},\quad R>\xi>0.
\end{equation*}
Remembering that $\displaystyle \lambda_{m,p}\paren{\xi_{m}}=E$, we get 
\begin{equation*}
\label{eq_fond_min}
\abs{\lambda_{m}'\paren{\xi_{m}}}\geqslant\frac{2k_{m}}{R^{3}}\left[1-\frac{E}{\paren{R-\xi_{m}}^{2}}\right],\quad R>\xi_{m}.
\end{equation*}\par
Let us choose $R=R_{m}:=\xi_{m}+\sqrt{2E}>\xi_{m}$. in order to obtain $\displaystyle E\paren{R_{m}-\xi_{m}}^{-2}=1/2$. 
This implies that $\displaystyle\abs{\lambda_{m}'\paren{\xi_{m}}}\geqslant k_{m}R_{m}^{-3}$. 
Observe that $R_{m}\sim \xi_{m}$ as $m\to+\infty$. Therefore Proposition \ref{loc_fr} provides 
\begin{equation*}
\label{rel_min_der}
\exists K>0\vi\exists M\in\mathbb{N}\vi\forall m\geqslant M,\quad\abs{\lambda_{m}'\paren{\xi_{m}}}\geqslant\frac{K}{\sqrt{k_{m}}}.
\end{equation*}
\paragraph{Upper bound}\ \\
Recall that $\Phi_{m}$ is defined by the formula \eqref{defi_phi_m_p}. Let us define the function $\Psi_{m}$ by
\begin{equation*}
\label{def_psi}
\Psi_{m}\paren{r}=\frac{e^{-2\Phi_{m}\paren{r}}}{r^{3}},\quad r>0.
\end{equation*}\par
Let's combine Propositions \ref{der_fin_fct_band} and \ref{prop_est_agmon}. We get that for $m$ large enough
\begin{equation}
\label{app_agmon_der}
\abs{\lambda_{m}'\paren{\xi_{m}}}\leqslant
2K k_{m}\sup_{r\in\mathbb{R}_{+}}{\Psi_{m}\paren{r}}.
\end{equation}\par
Therefore it is enough to prove that there is an integer $M$ and a constant $K>0$ such that 
\begin{equation}
\label{sup_psi_m}
\forall m\geqslant M,\quad\norme{\Psi_{m}}_{\lp{\infty}{\mathbb{R}_{+}}}\leqslant\frac{K}{k_{m}\sqrt{k_{m}}}.
\end{equation}
First note that $\Phi_{m}\geqslant 0$. Therefore for any $r\in\mathbb{R}_{+}$, $\Psi_{m}\paren{r}\leqslant r^{-1}$, meaning that $\Psi_{m}\paren{r}\to 0$ as $r\to+\infty$. Moreover, according to Lemma \ref{lemme_behavior_phi}, $-2\phi_{m}\paren{r}=2\alpha\ln\paren{r}+O\paren{1}$ as $r\to 0$. By combining it with the definition of $\Psi_{m}$, we deduce that $\Psi_{m}\paren{r}=O\paren{r^{2\alpha-3}}$ as $r\to 0$. 
Remembering that $\displaystyle \alpha>3/2$, we conclude that $\Psi_{m}\paren{r}\to 0$ as $r\to 0$. Hence we deduce that 
\begin{equation*}
\exists \tilde{r}_{m}>0,\quad\norme{\Psi_{m}}_{\lp{\infty}{\mathbb{R}}}=\Psi_{m}\paren{\tilde{r}_{m}}.
\end{equation*} 
Furthermore, $\tilde{r}_{m}$ is a critical point of $\Psi_{m}$. Therefore $\Psi_{m}'\paren{\tilde{r}_{m}}=0$ implies that
\begin{equation}
\label{eq_pt_critic}
\Phi_{m}'\paren{\tilde{r}_{m}}=-\frac{3}{2\tilde{r}_{m}}.
\end{equation} 
Observing that $\tilde{r}_{m}>0$, we get $\Phi_{m}'\paren{\tilde{r}_{m}}<0$. Remembering that $\Phi_{m}$ is non decreasing on $\intgd{\inf\paren{I_{m}}}{+\infty}$, we deduce that $\tilde{r}_{m}<\inf\paren{I_{m}}$. 
Note that $\inf\paren{I_{m}}$ is solution of $V_{m}\paren{r}=E$. Therefore Proposition \ref{loc_extr} provides a constant $K_{+}>0$ such that 
\begin{equation*}
\tilde{r}_{m}\leqslant\inf\paren{I_{m}}\leqslant K_{+}\sqrt{k_{m}}.
\end{equation*} \par
Now combine equations \eqref{eq_pt_critic} and \eqref{eq_eikon}. It yields
\begin{equation*}
\delta_{m}^{2}\paren{\frac{k_{m}}{\tilde{r}_{m}^{2}}+\paren{\tilde{r}_{m}-\xi_{m}}^{2}-E}=\frac{9}{4\tilde{r}_{m}^{2}}.
\end{equation*}\par
Hence we get 
\begin{equation}
\label{eq_fin_der_as}
\frac{\delta_{m}^{2}k_{m}-\frac{9}{4}}{\tilde{r}_{m}^{2}}=
\delta_{m}^{2}\paren{E-\paren{\tilde{r}_{m}-\xi_{m}}^{2}}\leqslant \delta_{m}^{2}E.
\end{equation}\par
Remembering that $\delta_{m}\sqrt{k_{m}}=\alpha$, estimate \eqref{eq_fin_der_as} can be written as 
$\displaystyle\paren{\alpha^{2}-9/4}\tilde{r}_{m}^{-2}\leqslant E\alpha^{2}k_{m}^{-1}$. 
Moreover $\displaystyle\alpha>3/2$, so there is a constant $K_{-}>0$ such that for $m$ large enough, 
\begin{equation}
\label{ctrl_r_m_tilde}
\tilde{r}_{m}\geqslant K_{-}\sqrt{k_{m}}. 
\end{equation}
Recall that $\Phi_{m}\paren{\tilde{r}_{m}}\geqslant 0$ meaning that 
$\norme{\Psi_{m}}_{\lp{\infty}{\mathbb{R}_{+}}}=\Psi\paren{\tilde{r}_{m}}\leqslant \tilde{r}_{m}^{-3}$. 
Combine it with the estimate \eqref{ctrl_r_m_tilde} and recall that $\Psi_{m}\paren{r}\to 0$ as $r\to 0$ and as $r\to+\infty$. 
It provides the estimate \eqref{sup_psi_m}. 
Finally we combine the estimates \eqref{app_agmon_der} and \eqref{sup_psi_m} that proves the upper bound.\par
\section{Velocity operator}
\label{sec_ope_cour}
In this section we assume that $n\geqslant 4$. 
The case $n=3$ could also been studied but according to remark \ref{rem_calc_ord_1}, attained thresholds arise in that case. 
We apply the results of the previous section to derive some properties of the current operator.\par
We refer to Section \ref{sect_oper_hami} for notations. 
Remember that $\mathcal{F}$ denotes the partial Fourier transform with respect to $x_{n}$. 
Let $\paren{r,\omega}$ be the cylindrical coordinates of $\mathbb{R}^{n-1}$, 
namely, for any $x\in\mathbb{R}^{n-1}\backslash \{0\}$, $r=\norme{x}_{2}$ and $\omega= r^{-1}x\in\mathbb{S}^{n-2}$. In terms of these variables, $\lp{2}{\mathbb{R}^{n-1}}=\lp{2}{\mathbb{R}_{+}\times\mathbb{S}^{n-2};r^{n-2}dr}$. Let $Y_{m,j}$, $m\geqslant 0$, $j\in\entier{1}{N_{m}}$ be the family of the spherical Harmonics. Remember that these functions form an orthonormal basis of solutions for the equation $-\Delta_{\mathbb{S}^{n-2}}u=\mu_{m}u$, $u\in\lp{2}{\mathbb{S}^{n-2}}$ and denote by $v_{m,p}\paren{\cdot,\xi}$ the eigenfunctions of $H_{m}\paren{\xi}$, $m\geqslant 0$ and $\xi\in\mathbb{R}$.\par
We define the $\paren{m,j,p}$-th generalized Fourier coefficient of $\varphi\in\lp{2}{\mathbb{R}^{n}}$ as
\begin{equation*}
\varphi_{m,j,p}\paren{\xi}:=\dfrac{1}{\sqrt{2\pi}}\int_{\mathbb{R}^{n-1}}{\widehat{\varphi}\paren{r,\omega,\xi}
\overline{Y_{m,j}\paren{\omega}v_{m,p}\paren{r,\xi}}r^{n-2}drd\omega},\quad \varphi\in\lp{2}{\mathbb{R}^{n}}.
\end{equation*}
Moreover for every $m\geqslant 0$, $j\in\entier{1}{N_{m}}$ and $p\in\mathbb{N}$, 
denote by $\pi_{m,j,p}$ the orthogonal projection associated with the $\paren{m,j,p}$-th harmonic 
and by $\pi_{p}$, the projection associated with all the harmonic that have $p$ as level:
\begin{equation*}
\pi_{m,j,p}\paren{\varphi}\paren{x}:=\frac{1}{\sqrt{2\pi}}\int_{\mathbb{R}}{e^{i\xi x_{n}}\varphi_{m,j,p}\paren{\xi}Y_{m,j}\paren{\omega}v_{m,p}\paren{r,\xi}d\xi},\quad x\in\mathbb{R}^{n},
\end{equation*}
\begin{equation*}
\pi_{p}:=\sum_{m=0}^{+\infty}{\sum_{j=1}^{N_{m}}{\pi_{m,j,p}}}.
\end{equation*}
In light of Section \ref{sect_oper_hami}, every $\varphi\in\lp{2}{\mathbb{R}^{n}}$ is decomposed as
\begin{eqnarray*}
\varphi = \sum_{m=0}^{+\infty}{\sum_{j=1}^{N_{m}}{\sum_{p=1}^{+\infty}{\pi_{m,j,p}\paren{\varphi}}}}
= \sum_{p=1}^{+\infty}{\pi_{p}\paren{\varphi}}.
\end{eqnarray*}
Moreover the Parseval theorem yields
\begin{equation}
\label{pars_id}
\norme{\varphi}^{2}_{2}=
\sum_{m=0}^{+\infty}{
\sum_{j=1}^{N_{m}}{
\sum_{p=1}^{+\infty}{\norme{\varphi_{m,j,p}}_{2}^{2}}}.
}
\end{equation}\par
Finally for any non-empty interval $I\subset\mathbb{R}$, denote by $\mathbb{P}_{I}$ the spectral projection of $H$ associated with $I$. 
A quantum state $\varphi\in\lp{2}{\mathbb{R}^{n}}$ is said to be concentrated in $I$ if $\mathbb{P}_{I}\varphi=\varphi$. With reference to Section \ref{sect_oper_hami}, this condition can be written as
\begin{equation}
\label{caract_proj_spec_fct_band}
\forall m\geqslant 0\vi\forall j\in\entier{1}{N_{m}}\vi\forall p\geqslant 1,\quad \text{supp}\paren{\varphi_{m,j,p}}\subset\lambda_{m,p}^{-1}\paren{I}.
\end{equation}\par
Let $x_{n}$ be the position operator defined as the multiplier by coordinate $x_{n}$ in $\lp{2}{\mathbb{R}^{n}}$:
\begin{equation*}
\paren{x_{n}f}x = x_{n} f\paren{x},\quad x=\paren{x_{1},\cdots,x_{n}}\in\mathbb{R}^{n}
\end{equation*}
and let $x_{n}\paren{t}$ be the Heisenberg variable defined as
\begin{equation*}
x_{n}\paren{t}:=e^{itH}x_{n}e^{-itH}.
\end{equation*}\par
A quantum state $\varphi$ is a solution of the Schrödinger equation \eqref{eq_schrodinger}. 
Thus $\varphi\paren{x,t}=e^{-itH}\varphi\paren{x,0}$ and we deduce by a straightforward calculation that
\begin{equation}
\label{eq_just_time_evol_pos}
\scalaire{x_{n}\varphi\paren{\cdot,t}}{\varphi\paren{\cdot,t}}=
\scalaire{x_{n}\paren{t} \varphi\paren{\cdot,0}}{\varphi\paren{\cdot,0}},\quad t\in\mathbb{R}.
\end{equation}
Therefore the time evolution of the position operator $x_{n}$ is $x_{n}\paren{t}$ and its time derivative is the velocity, given by 
\begin{equation}
\label{eq_der_pos_op}
\partial_{t}x_{n}\paren{t}=ie^{itH}\left[H,x_{n}\right]e^{-itH}.
\end{equation}
We define the current operator $J$ as the following self-adjoint operator acting on $\domaine{H}\cap\domaine{x_{n}}$ such that the current carried by a state $\varphi$ is $\scalaire{J\varphi}{\varphi}$. 
\begin{equation}
\label{defi_vel_ope}
J:=-i\left[H,x_{n}\right]=-2\paren{i\partial_{x_{n}}+r}.
\end{equation}
Note that 
\begin{equation}
\label{eq_der_pos_op_bis}
\partial_{t}x_{n}\paren{t}=-e^{itH}Je^{-itH}.
\end{equation}\par
Since $\mathcal{F}$, is an isometry, we observe that $\mathcal{F}J\mathcal{F}^{-1} =-2\paren{r-\xi}$. 
Therefore the Feynman-Hellman formula (see equation \eqref{app_fey_hel_form}) yields
\begin{equation}
\label{app_fibr_cour}
\scalaire{J\pi_{p}\varphi}{\pi_{p}\varphi}=
\sum_{m=0}^{+\infty}{
\sum_{j=1}^{N_{m}}{
\int_{\mathbb{R}}{\lambda_{m,p}'\paren{\xi}\abs{\varphi_{m,j,p}\paren{\xi}}^{2}d\xi}
}
},
\quad p\in\mathbb{N}.
\end{equation}\par
In Theorem \eqref{theo_trans_courr}, we will combine this identity with Theorem \ref{theo_comp_asym_der} to control the current operator. 
We define for each $M\in\mathbb{N}$ and each $p\in\mathbb{N}$:
\begin{eqnarray*}
\displaystyle X_{I,M,p}^{-}:=&\left\{\varphi\in\text{Ran}\paren{\mathbb{P}_{I}}\cap\text{Ran}\paren{\pi_{p}}\vi
\forall m\geqslant M+1\vi\forall j\in\entier{1}{N_{m}}\vi
\varphi_{m,j,p}= 0\right\},\\
\displaystyle X_{I,M,p}^{+}:=&\left\{\varphi\in\text{Ran}\paren{\mathbb{P}_{I}} \cap\text{Ran}\paren{\pi_{p}}\vi
\forall m\leqslant M\vi\forall j\in\entier{1}{N_{m}}\vi
\varphi_{m,j,p}= 0\right\},\\
\displaystyle &X_{I,M}^{-}:=\displaystyle\bigoplus_{p=1}^{+\infty}{X_{I,M,p}^{-}};\\
\displaystyle &X_{I,M}^{+}:=\displaystyle\bigoplus_{p=1}^{+\infty}{X_{I,M,p}^{+}}.\\
\end{eqnarray*}
Note that $\text{Ran}\paren{\mathbb{P}_{I}}=X_{I,M}^{-}\oplus X_{I,M}^{+}$ and that these spaces are $H$ invariant.\par
\begin{theo}
\label{theo_trans_courr}
Let $I\subset\sigma\paren{H}$ be a non-empty interval such that $I\cap\left\{E_{p}\vi p\geqslant 1\right\}=\emptyset$. 
\begin{enumerate}
\item $\forall M\geqslant 0\vi\exists C^{-}>0\vi\forall\varphi\in X_{I,M}^{-},\quad
\abs{\scalaire{J\varphi}{\varphi}}\geqslant C^{-}\norme{\varphi}^{2}_{2}$
\item $\exists C_{+}\vi\exists M_{0}\geqslant 0\vi\forall M\geqslant M_{0},\quad\forall\varphi\in X_{I,M}^{+},\quad
\abs{\scalaire{J\varphi}{\varphi}}\leqslant \dfrac{C^{+}}{\sqrt{k_{M+1}}}\norme{\varphi}^{2}_{2}$.
\end{enumerate}
\end{theo}
\begin{proof}
First of all, observe that $I$ is bounded and recall that $E_{p}\to+\infty$ as $p\to +\infty$. 
Therefore $P_{I}:=\left\{p\in\mathbb{N},\quad E_{p}\geqslant\sup\paren{I}\right\}$ is a finite set. 
Moreover, remembering that for every $m\in\mathbb{N}$ and $p\in\mathbb{N}$, 
$\inf\left\{\lambda_{m,p}\paren{\xi}\in\mathbb{R},\quad\xi\in\mathbb{R}\right\} = E_{p}$, we get that
$P_{I}=\left\{p\in\mathbb{N}\vi\exists m\geqslant 0,\quad I\cap\lambda_{m,p}\paren{\mathbb{R}}\neq\emptyset\right\}$. 
Moreover notice that for every $p\in P_{I}$ and every $m\in\mathbb{N}$, $ I\cap\lambda_{m,p}\paren{\mathbb{R}}\neq\emptyset$. 
Therefore it is enough to prove the theorem for $\varphi\in X_{I,m,p}^{\pm}$ for a certain $p\in P_{I}$ fixed. 
We simplify the notations by ommiting the index $p$.\par
\textbf{Proof of the first part. }
Let $M\geqslant 0$ and let $\varphi\in X_{I,M}^{-}$. 
Note that $\varphi=\pi_{p}\varphi$. 
Therefore, according to the embedding \eqref{caract_proj_spec_fct_band} and to the identity \eqref{app_fibr_cour}, 
\begin{equation*}
\scalaire{J\varphi}{\varphi}=
\sum_{m=0}^{M}{\sum_{j=1}^{N_{m}}{
\int_{\lambda_{m}^{-1}\paren{I}}{\lambda_{m}'\paren{\xi}\abs{\varphi_{m,j}\paren{\xi}}^{2}d\xi}
}}.
\end{equation*}
Moreover $E_{p}\not\in\overline{I}$, thus for every $m\in\entier{0}{M}$, $\lambda^{-1}_{m}\paren{I}$ is bounded. 
Therefore, Proposition \ref{der_fin_fct_band} states that for every $m\in\entier{0}{M}$, 
$D_{m}:=\inf\left\{\lambda_{m}'\paren{\xi}\vi\lambda_{m}\paren{\xi}\in I\right\}>0$. 
Thus $C^{-}:=\inf\left\{D_{m},m\in\entier{0}{M}\right\}>0$.
Hence
\begin{equation}
\label{est_courr}
\abs{\scalaire{J\varphi}{\varphi}}\geqslant
C^{-}\sum_{m=0}^{M}{\sum_{j=1}^{N_{m}}{
\int_{\lambda_{m}^{-1}\paren{I}}{\abs{\varphi_{m,j}\paren{\xi}}^{2}d\xi}
}}.
\end{equation}
Remember that $\varphi_{m,j}$ is localized in $\lambda_{m}^{-1}\paren{I}$ (see the embedding \eqref{caract_proj_spec_fct_band}). Therefore, 
\begin{equation*}
\int_{\lambda_{m}^{-1}\paren{I}}{\abs{\varphi_{m,j}\paren{\xi}}^{2}d\xi}=
\int_{\mathbb{R}}{\abs{\varphi_{m,j}\paren{\xi}}^{2}d\xi}=
\norme{\varphi_{m,j}}^{2}_{2}.
\end{equation*}
Hence according to the Parseval's identity \eqref{pars_id}, 
\begin{equation}
\label{app_par_id}
\sum_{m=0}^{M}{\sum_{j=1}^{N_{m}}{
\int_{\lambda_{m}^{-1}\paren{I}}{\abs{\varphi_{m,j}\paren{\xi}}^{2}d\xi}
}}=\norme{\varphi}_{2}^{2}.
\end{equation}
We combine it with the estimate \eqref{est_courr} that provides the first statment of the Theorem.\par
\textbf{Proof of the second part. } 
Let $\varphi\in X_{I,M}^{+}$. 
We prove in the same way as for the first part that 
\begin{equation*}
\scalaire{J\varphi}{\varphi}=
\sum_{m=M+1}^{+\infty}{\sum_{j=1}^{N_{m}}{
\int_{\lambda_{m}^{-1}\paren{I}}{\lambda_{m}'\paren{\xi}\abs{\varphi_{m,j}\paren{\xi}}^{2}d\xi}
}}.
\end{equation*}\par
Therefore, according to Theorem \ref{theo_comp_asym_der}, there exist $M_{0}\geqslant 0$ and $C_{+}>0$ such that for every $M\geqslant M_{0}$, 
\begin{equation}
\label{est_courr_2}
\abs{\scalaire{J\varphi}{\varphi}}\leqslant
\sum_{m=M+1}^{+\infty}{\sum_{j=1}^{N_{m}}{
\frac{C_{+}}{\sqrt{k_{m}}}
\int_{\lambda_{m}^{-1}\paren{I}}{\abs{\varphi_{m,j}\paren{\xi}}^{2}d\xi}
}}.
\end{equation}
Observe that for $m\geqslant M+1$, $C_{+}k^{-1/2}_{m}\leqslant C_{+}k^{-1/2}_{M+1}$. 
We combine it with the estimate \eqref{est_courr_2} and with the Parseval's identity \eqref{pars_id} that yields
\begin{equation*}
\abs{\scalaire{J\varphi}{\varphi}}\leqslant
\frac{C_{+}}{\sqrt{k_{M+1}}}\norme{\varphi}_{2}^{2}.
\end{equation*}
\end{proof}
\begin{remarque}
\label{rem_faibl_courr}
Remember that $k_{M}\to+\infty$ as $M\to+\infty$. According to Theorem \ref{theo_trans_courr}, 
for every $\varepsilon>0$ and for any bounded energy interval $I\subset\sigma\paren{H}$, 
there is some quantum state $\varphi_{\varepsilon}\in\text{Ran}\paren{\mathbb{P}_{I}}$ 
such that $\norme{\varphi_{\varepsilon}}=1$ and $\abs{\scalaire{J\varphi_{\varepsilon}}{\varphi_{\varepsilon}}}\leqslant\epsilon$, even if $\overline{I}$ is away from the Landau levels.\par
\end{remarque}

\normalsize{}
\bibliographystyle{alpha}
\addcontentsline{toc}{section}{\bibname}
\bibliography{biblio}
\end{document}